\documentclass[11pt]{article}
\usepackage[T1]{fontenc}
\usepackage{lmodern,amsmath,amsthm,amsfonts,amssymb,graphicx,float,microtype,thmtools,underscore,mathtools,anyfontsize}
\usepackage[titletoc,title]{appendix}
\usepackage[usenames,dvipsnames,svgnames,table]{xcolor}
\usepackage[unicode=true]{hyperref}
\usepackage{nicefrac}
\hypersetup{%
colorlinks,
linkcolor={black},
breaklinks=true,
citecolor={black},
urlcolor={blue!60!black},
pdftitle={Treewidth, Hadwiger number, and Induced Minors},
pdfauthor={Campbell--Davies--Distel--Fredrickson--Gollin--Hendrey--Hickingbotham--Wiederrecht--Wood--Yepremyan}
}
\usepackage{xcolor}
\usepackage{verbatim}
\usepackage[noabbrev,capitalise]{cleveref}
\crefname{lem}{Lemma}{Lemmas}
\crefname{thm}{Theorem}{Theorems}
\crefname{cor}{Corollary}{Corollaries}
\crefname{prop}{Proposition}{Propositions}
\crefname{conj}{Conjecture}{Conjectures}
\crefname{rmk}{Remark}{Remarks}
\crefname{openproblem}{Open Problem}{Open Problems}
\crefformat{equation}{(#2#1#3)}
\Crefformat{equation}{Equation #2(#1)#3}
\usepackage[shortlabels]{enumitem}
\setlist[itemize]{topsep=0ex,itemsep=0ex,parsep=0.4ex}
\setlist[enumerate]{topsep=0ex,itemsep=0ex,parsep=0.4ex}
\newcommand{\defn}[1]{\textcolor{Maroon}{\emph{#1}}}

\newcommand{\GG}{\mathcal{G}}

\newcommand{\NN}{\mathbb{N}}

\usepackage[longnamesfirst,numbers,sort&compress]{natbib}
\makeatletter
\def\NAT@spacechar{~}
\makeatother
\usepackage[bmargin=25mm,tmargin=25mm,lmargin=35mm,rmargin=35mm]{geometry}
\allowdisplaybreaks

\renewcommand{\geq}{\geqslant}
\renewcommand{\leq}{\leqslant}

\DeclareMathOperator{\polylog}{polylog}

\DeclareMathOperator{\tw}{tw}

\DeclareMathOperator{\igrid}{igrid}

\DeclareMathOperator{\had}{had}

\newcommand*{\boundary}{\mathsf{bd}}
\newcommand*{\Root}{\mathsf{root}}
\DeclareMathOperator{\GF}{GF}

\definecolor{ao(english)}{rgb}{0.0, 0.5, 0.0}

\renewcommand{\thefootnote}{\fnsymbol{footnote}}
\allowdisplaybreaks
\theoremstyle{plain}
\newtheorem{thm}{Theorem}
\newtheorem{lem}[thm]{Lemma}
\newtheorem{cor}[thm]{Corollary}

\newtheorem{openproblem}[thm]{Open Problem}
\theoremstyle{definition}
\newtheorem{conj}[thm]{Conjecture}
\newtheorem*{claim}{Claim}

\newenvironment{subproof}[1][Proof.]{%
    \begin{proof}[{#1}]%
        }{%
    \end{proof}}

\begin{document}

\title{\bf\fontsize{16pt}{0pt}\selectfont  Treewidth, Hadwiger Number, and Induced Minors}

\author{
Rutger Campbell\,\footnotemark[1] \qquad
James Davies\,\footnotemark[3] \qquad
Marc Distel\,\footnotemark[2] \qquad
Bryce Frederickson\,\footnotemark[4] \\
J.~Pascal Gollin\,\footnotemark[5] \qquad
Kevin Hendrey\,\footnotemark[2] \qquad
Robert~Hickingbotham\,\footnotemark[7] \\
Sebastian Wiederrecht\,\footnotemark[6] \qquad
David~R.~Wood\,\footnotemark[2]\qquad
Liana Yepremyan\,\footnotemark[4] 
}

\footnotetext[1]{Discrete Mathematics Group, Institute for Basic Science (IBS), Daejeon, Republic~of~Korea (\texttt{rutger@ibs.re.kr}). Research of R.C.\ supported by the Institute for Basic Science (IBS-R029-C1). }

\footnotetext[3]{University of Cambridge, Cambridge, UK (\texttt{jgd37@cam.ac.uk}).}

\footnotetext[2]{School of Mathematics, Monash University, Melbourne, Australia (\texttt{\{marc.distel,kevin.hendrey1,david.wood\}@monash.edu}). Research of D.W.\ supported by the Australian Research Council. Research of M.D.\ supported by Australian Government Research Training Program Scholarships. Research of K.H.\ supported by the Institute for Basic Science (IBS-R029-C1) and by the Australian Research Council.}

\footnotetext[4]{Department of Mathematics, Emory University, USA (\texttt{\{bryce.frederickson,liana.yepremyan\}@emory.edu})}

\footnotetext[5]{FAMNIT, University of Primorska, Slovenia (\texttt{pascal.gollin@famnit.upr.si}). Research of J.P.G.\ supported by the Institute for Basic Science (IBS-R029-Y3) and in part by the Slovenian Research and Innovation Agency (research project N1-0370).}

\footnotetext[6]{KAIST, School of Computing, South Korea (\texttt{swiederrecht@kaist.ac.kr}). Reserach of S.W.\ supported by supported by the Institute for Basic Science (IBS-R029-C1) and the KAIST Start-Up Grant}

\footnotetext[7]{Univ. Lyon, ENS de Lyon, UCBL, CNRS, LIP, France (\texttt{robert.hickingbotham@ens-lyon.fr}). Research of R.H.\ partially supported by Australian Government Research Training Program Scholarships.}

\sloppy
	
\maketitle

\begin{abstract}
    Treewidth and Hadwiger number are two of the most important parameters in structural graph theory. 
    This paper studies graph classes in which large treewidth implies the existence of a large complete graph minor. 
    To formalise this, we say that a graph class~$\GG$ is ${(\tw,\had)}$-bounded if there is a function~$f$ (called the ${(\tw,\had)}$-bounding function) such that~${\tw(G) \leq f(\had(G))}$ for every graph~${G \in \GG}$. 
    We characterise ${(\tw,\had)}$-bounded graph classes as those that exclude some planar graph as an induced minor, and use this characterisation to show that every proper vertex-minor-closed class is ${(\tw,\had)}$-bounded. Furthermore, we demonstrate that any ${(\tw,\had)}$-bounded graph class has a ${(\tw,\had)}$-bounding function in~${O(\had(G)^9\polylog(\had(G)))}$. Our bound comes from the bound for the Grid Minor Theorem given by Chuzhoy and Tan, and any quantitative improvement to their result will lead directly to an improvement to our result.
    
   More strongly, we conjecture that every ${(\tw,\had)}$-bounded graph class has a linear ${(\tw,\had)}$-bounding function. 
    In support of this conjecture, we show that it holds for the class of outer-string graphs, and for a natural generalisation of outer-string graphs: intersection graphs of strings rooted at the boundary of a fixed surface. We also verify our conjecture for low-rank perturbations of circle graphs, which is an important step towards verifying it for all proper vertex-minor-closed classes. 
\end{abstract}



\newpage

\section{Introduction}
\label{sec:intro}

\renewcommand{\thefootnote}{\arabic{footnote}}

\footnotetext[1]{See Section~\ref{sec:prelims} for omitted definitions.}
\setcounter{footnote}{1}

Treewidth\footnotemark[1] is the standard measure of how similar a graph is to a tree, and is one of the most important parameters in structural graph theory. The celebrated Grid Minor Theorem of \citet{RS-V} states that every graph with large treewidth has a large grid minor. Since grids themselves have large treewidth and every planar graph is a minor of some grid, 
for any graph $H$, there is an integer $c_H$ such that every $H$-minor-free graph has treewidth at most $c_H$ if and only if $H$ is planar.
The situation can be different for $H$-minor-free graphs within a fixed hereditary class.
For example, \citet{HIMW24} showed that within the class of circle graphs (the intersection graphs of chords of a circle), $K_t$-minor-free graphs have treewidth less than~${12t}$. The \defn{Hadwiger number~$\had(G)$} of a graph $G$ is the maximum integer $t$ such that $K_t$ is a minor of $G$.
It turns out that there is a dichotomy between hereditary classes in which large treewidth implies large Hadwiger number, and hereditary classes that  contain planar graphs of arbitrarily large treewidth. 
This can be seen by combining the Grid Minor Theorem with the following folklore result (and it is a true dichotomy, since planar graphs have Hadwiger number at most~$4$).

\begin{thm}
    \label{thm:treewidthplanarminor}
    There is a function~${f \colon \mathbb{N}^2 \to \mathbb{N}}$ such that for any non-negative integers~$k$ and $t$, every graph which contains an~${f(k,t) \times f(k,t)}$-grid as a minor either contains~$K_t$ as a minor or contains a planar induced subgraph with treewidth at least $k$.
\end{thm}

Several variations and strengthenings of this result can be found in the literature (see for example \cite{AAKST21,AACHS24,Bonnet24}), but little effort has been made to optimise the function~$f$. 
The current best known bounds for Robertson and Seymour's Flat Wall Theorem~\cite{RS-XIII} were given by Kawarabayashi, Thomas and Wollan~\cite{KTW18}, and imply \Cref{thm:treewidthplanarminor} with~${f(k,t) \in O(t^{26}+t^{24}k)}$ (much better than any bound implicit in the other results mentioned above). 
Our first result is to prove \Cref{thm:treewidthplanarminor} with~${f(k,t) \in O(kt)}$. 
This feeds directly into the primary goal of this paper: 
to better understand the quantitive relationship between Hadwiger number and treewidth in hereditary classes for which such a relationship exists. 
To formalise this, we introduce the following notions. 

We say a hereditary graph class~$\GG$ is \defn{${(\tw,\had)}$-bounded} if there is a function~$f$ (called a \defn{$(\tw,\had)$-bounding function of $\GG$}) such that~${\tw(G) \leq f(\had(G))}$ for every graph~${G \in \GG}$. 
As discussed above, a hereditary class~$\GG$ is ${(\tw,\had)}$-bounded if and only if the planar graphs in~$\GG$ have bounded treewidth. 
For example, planar graphs are not ${(\tw,\had)}$-bounded, whereas chordal graphs are, since~${\tw(G) = \had(G)-1}$ for every chordal graph~$G$. 
Other examples of ${(\tw,\had)}$-bounded classes include all proper vertex-minor-closed classes, as well as various geometrically defined classes, as we discuss later. 

Our next theorem summarises and develops the current state of knowledge regarding general ${(\tw, \had)}$-bounded graph classes and their ${(\tw,\had)}$-bounding functions. 
As mentioned above,
${(\tw, \had)}$-bounded classes are characterised as hereditary classes that do not contain planar graphs of arbitrarily large treewidth. 
In terms of induced subgraphs, this yields
an infinite list of obstructions. 
Here we provide a concise characterisation in the language of induced minors. 
A graph~$H$ is an \defn{induced minor} of a graph~$G$ if a graph
isomorphic to~$H$ can be obtained from~$G$ by deleting vertices and contracting edges. 
We say a graph class~$\GG$ \defn{excludes~$H$ as an induced minor} if no graph in~$\GG$ contains~$H$ as an induced minor. 

\begin{restatable}{thm}{NewMainEquivalent}\label{MainEquivalent1}
    Let~$\GG$ be a hereditary graph class and let~${g_{\GG} \colon \mathbb{N}\to \mathbb{N}}$ be a function such that for every positive integer~$k$, every graph in~$\GG$ with treewidth at least~${g_{\GG}(k)}$ contains a ${(k \times k)}$-grid as a minor. 
    The following statements are equivalent. 
    \begin{enumerate}[{(1)}]
        \item\label{(1)} $\GG$ is ${(\tw,\had)}$-bounded; 
        \item\label{(2)} there is a constant~$c$ such that $f(t)\coloneqq g_{\GG}(ct)$ is a $(\tw,\had)$-bounding function of~$\GG$; 
        \item\label{(3)} 
        there is a $(\tw,\had)$-bounding function $f$ for~$\GG$ with $f(t)\in O(t^9\polylog t)$; and 
        \item\label{(4)} $\GG$ excludes a planar graph as an induced minor. 
    \end{enumerate}
\end{restatable}

Note that the ${(k \times k)}$-grid is a minor of~$K_{k^2}$, so every ${(\tw,\had)}$-bounding function for a class~$\mathcal{G}$ must be in~${\Omega(\sqrt{g_{\GG}})}$ for some function~$g_{\GG}$ as above. 
\cref{MainEquivalent1} provides a clear strategy for finding ${(\tw,\had)}$-bounding functions: 
determine the smallest function~$g$ such that every graph of treewidth~${g(k)}$ contains a ${(k\times k)}$-grid minor. 
Robertson, Seymour, and Thomas~\cite{RST94} showed the current best known lower bound~${g(k) \in \Omega(k^2 \log k)}$, and speculated that this bound might be tight. 
Demaine, Hajiaghayi, and Kawarabayashi~\cite{DHK-Algo09} conjectured that~${g(k) \in \Theta(k^3)}$ based on their analysis of map graphs.
Since \cref{MainEquivalent1} allows us to restrict to the graphs in the hereditary class~$\GG$ when defining the function~$g_{\GG}$, it is feasible that in some cases our result may lead to ${(\tw,\had)}$-bounding functions in~${o(t^2\log t)}$. 
However, it is worth noting that since~$K_{k^2-1}$ has treewidth~${k^2-2}$ and contains no ${(k \times k)}$-grid minor, we can not obtain a subquadratic $(\tw,\had)$-bounding function from \cref{MainEquivalent1} for any hereditary class with unbounded clique number (for example). However, we believe this is simply a limitation of our methodology.
A hereditary graph class~$\GG$ is \defn{linearly $(\tw,\had)$-bounded} if there is a constant~$c$ such that ${\tw(G)\leq c\,\had(G)}$ for every~${G \in \GG}$.
Recall that~\citet{HIMW24} showed that the class of circle graphs, which has unbounded clique number, is linearly $(\tw,\had)$-bounded. Partly motivated by this, we conjecture the following.

\begin{conj}\label{ConjMainLinear}
    Every ${(\tw,\had)}$-bounded hereditary graph class is linearly ${(\tw,\had)}$-bounded.
\end{conj}

In support of \cref{ConjMainLinear}, we show that several natural ${(\tw,\had)}$-bounded graph classes are linearly ${(\tw,\had)}$-bounded. 
First, we consider outer-string graphs. 
An \defn{outer-string graph} is an intersection graph of curves on a disk where each curve has an endpoint on the boundary of the disk; see \Cref{fig:houseouterstring} for an example. 
This class generalises circle graphs. 
\begin{figure}[ht]
    \centering
    \includegraphics[scale=0.8]{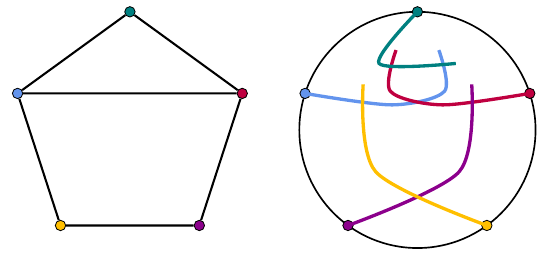}
    \caption{An outer-string graph (left) with a representation (right). }
    \label{fig:houseouterstring}
\end{figure}

\begin{thm}\label{thm:outerstring}
    For every outer-string graph~$G$,
    \[
        \tw(G) \leq 15 \had(G)-2. 
    \]
\end{thm}

We also consider a generalisation of outer-string graphs to surfaces with higher genus. 
The \defn{Euler genus} of a surface obtained from a sphere by adding~$h$ handles and~$k$ crosscaps is~${2h+k}$. 
We say that a graph is \defn{${(g,c)}$-outer-string} if it is an intersection graph of curves on a surface~$\Sigma$ obtained from a surface of Euler genus~$g$ by deleting the interior of~$c$ disjoint closed discs, and each curve has an endpoint on the boundary of~$\Sigma$. 
Note that ${(0,1)}$-outer-string graphs are precisely outer-string graphs. 

\begin{restatable}{thm}{outerstringgeneral}\label{thm:outerstringgeneral}
    For any integers~${c \geq 1}$ and~${g \geq 0}$, for every ${(g,c)}$-outer-string graph~$G$,
    \[
        \tw(G) \leq 10^3 ( 2 g^{5/2} c^2 + c ) \had(G). 
    \]
\end{restatable}

Finally, we consider vertex-minor-closed classes in general, although our analysis takes an indirect approach. 
As discussed in \cref{SectionVertexMinor}, every proper vertex-minor-closed class excludes some planar graph as an induced minor, and so is ${(\tw,\had)}$-bounded by \Cref{MainEquivalent1}. 
For an integer~${r \geq 0}$, a graph~$\tilde{G}$ is a \defn{rank-$r$-perturbation} of another graph~$G$ if the binary adjacency matrix of~$\tilde{G}$ is obtained from the sum of a binary adjacency matrix of~$G$ with a ${(\lvert V(G) \rvert \times \lvert V(G) \rvert)}$-matrix of rank~$r$, which we call the \defn{perturbation matrix}, by changing all diagonal entries to~$0$. 
Low-rank perturbations of circle graphs are fundamental to the study of vertex-minors, and form the key ingredient for a structural conjecture of Geelen for proper vertex-minor-closed classes (see \citep[Section~1.5]{McCarty21}). 
As a step towards verifying \cref{ConjMainLinear} for proper vertex-minor-closed classes, we show that low-rank perturbations of circle graphs are linearly ${(\tw,\had)}$-bounded.

\begin{thm}\label{thm:c-pertubationofcirclegraph}
    For every graph~$G$ that is a rank-$r$-perturbation of a circle graph,
    \[ 
        \tw(G) \leq 65 \cdot 2^{3r} \had(G). 
    \]
\end{thm}

Our work ties in to the wider study of graph classes that exclude a planar graph as an induced minor. 
Recently, several important algorithmic and structural conjectures have been made for such graphs. 
Let~$\GG$ be an arbitrary ${(\tw,\had)}$-bounded class. 
On the algorithmic side, \citet{DMS24a} and \citet{Gartland23} independently conjectured that many problems that are NP-hard in general 
such as \textsc{Maximum Weight Independent Set}, \textsc{Feedback Vertex Set}, and \textsc{$3$-Colourability} are polynomial-time solvable on graphs in $\GG$. 
On the structural side, \citet{Gartland23} conjectured that every graph in~$\GG$ has a balanced separator dominated by a bounded number of vertices. 
A similar, Coarse Grid Minor Conjecture, due to \citet{GP23}, would imply that graphs in~$\GG$ are quasi-isometric to graphs of bounded treewidth. 
Thus, there is substantial interest in understanding the structure of graphs that exclude a planar graph as an induced minor. 
Recently, a large body of work has emerged attempting to characterise the unavoidable induced subgraphs of graphs with large treewidth that exclude some planar graph as an induced minor \citep{LR22,ACV22,ACHS23,AACHSV24,AACHSV24a,ACDHRSV24,ACHS22,PSTT21,ST21,BBDEGHTW24,AACHS24,ACHS25}.
Our work has direct implications for this project, giving explicit lower bounds on the complete graph minors that these obstructions must contain.

We now present a brief overview of the paper. 
In \cref{sec:prelims} we provide some preliminary definitions and results. 
In \cref{sec:InducedGridMinors} we prove \cref{MainEquivalent1}. 
In \cref{SectionVertexMinor} we discuss the ${(\tw,\had)}$-boundedness of proper vertex minor-closed classes. 
In \cref{sec:orderedgraphs}, we prove \cref{thm:outerstring} by proving a stronger result in the setting of ordered graphs. 
In \cref{sec:outerstring} we prove \cref{thm:outerstringgeneral}. 
Finally, in \cref{sec:circleperturbations} we prove \cref{thm:c-pertubationofcirclegraph}.

\section{Preliminaries}
\label{sec:prelims}

For integers~${m,n \in \mathbb{Z}}$, 
let~${[m,n] \coloneqq \{ z \in \mathbb{Z} \colon m \leq z \leq n\}}$ and~${[n] \coloneqq [1,n]}$. 
Let~$\NN$ be the set of non-negative integers.

All graphs in this paper are finite and simple unless specified otherwise. A \defn{graph class}~$\GG$ is a class of graphs closed under isomorphism. A graph class~$\GG$ is \defn{proper} if some graph is not in~$\GG$. A graph class~$\GG$ is \defn{hereditary} if it is closed under induced subgraphs. 

A graph~$H$ is a \defn{minor} of a graph~$G$ if a graph isomorphic to~$H$ can be obtained from~$G$ by contracting edges and deleting vertices and edges. 
A graph class~$\GG$ is \defn{minor-closed} if for every graph~${G \in \GG}$ every minor of~$G$ is in~$\GG$.

For a vertex~$v$ of a graph $G$, the graph obtained from~$G$ by \defn{locally complementing} at $v$ is the graph \defn{$G*v$} obtained by complementing the subgraph induced by the neighbourhood of~$v$ in $G$. 
A graph~$H$ is a \defn{vertex-minor} of a graph~$G$ if $H$ is isomorphic to a graph that can be obtained from~$G$ by a sequence of local complementations and vertex deletions.

\subsection{Treewidth and Separators}
\label{subsec:balancedSepsForSets}

 A \defn{tree-decomposition} of a graph~$G$ is a pair ${(T,\beta)}$ such that:
 \begin{itemize}
     \item $T$ is a tree and ${\beta \colon V(T) \to 2^{V(G)}}$ is a function, 
     \item for every edge ${vw \in E(G)}$, there exists a node ${x \in V(T)}$ with ${v,w \in \beta(x)}$, and 
     \item for every vertex ${v \in V(G)}$, the set $\{ x \in V(T) \colon v \in \beta(x) \}$ induces a non-empty (connected) subtree of~$T$. 
 \end{itemize}
 The \defn{width} of a tree-decomposition~${(T,\beta)}$ is ${\max\{ \lvert \beta(x)\rvert \colon x \in V(T) \}-1}$. 
 The \defn{treewidth~$\tw(G)$} of a graph~$G$ is the minimum width of a tree-decomposition of~$G$. 
 Treewidth is the standard measure of how similar a graph is to a tree. 
 Indeed, a connected graph has tree-width at most~$1$ if and only if it is a tree. 

\citet{RS-V} first showed there is a function~$g$ such that every graph of treewidth at least~${g(k)}$ contains a ${(k \times k)}$-grid as a minor. 
The function~$g$ has since been steadily improved~\citep{CC16,CT21,LeafSeymour15,RST94,KK20}. 
The best known bounds are due to \citet{CT21}.

\begin{thm}[\citet{CT21}]\label{thm:polygrid}
    There exist integers~$c_1$ and $c_2$ such that for every positive integer~$k$, every graph of treewidth at least~${c_1k^9\log^{c_2}k}$ contains a ${(k \times k)}$-grid as a minor.
\end{thm}

Treewidth is closely connected to separators, as we now explain. 
Let~${\alpha \in [\frac{1}{2},1)}$, $G$ be a graph, and~${X \subseteq V(G)}$ be a set of vertices. 
A \defn{separation} of a graph~$G$ is a pair~${(A,B)}$ of subsets of~${V(G)}$ such that~${A \cup B = V(G)}$ and no edge of~$G$ has one end in~${A \setminus B}$ and the other in~${B \setminus A}$. 
The \defn{order} of the separation~${(A,B)}$ is~${|A \cap B|}$. 
For~${X \subseteq V(G)}$, a separation~${(A,B)}$ of~$G$ is \defn{$\alpha$-balanced for~$X$} if ${\big\lvert X\cap  (A\setminus B) \big\rvert\leq \alpha \cdot \lvert X\rvert}$ and ${\big\lvert X\cap  (B\setminus A) \big\rvert\leq \alpha \cdot \lvert X\rvert}$. 
The following simple condition is sufficient for a separation to be $\alpha$-balanced.

\begin{lem}
    \label{obs:balancedseps}
    If~${(A,B)}$ is a separation of a graph~$G$, and~${X \subseteq V(G)}$ with~${|X| \leq c|A\cap X|}$ and~${|X| \leq c|B \cap X|}$ for some~${c>1}$, then~${(A,B)}$ is~$\frac{c-1}{c}$-balanced for~$X$. 
\end{lem}

\begin{proof}
    Let~${X_A \coloneqq X \cap (A\setminus B)}$, and~${X_B \coloneqq X\cap (B\setminus A)}$, and~${X^* \coloneqq X \cap A \cap B}$. 
    So~${X_A,X_B,X^*}$ is a partition of~$X$. 
    By assumption, 
    \[
        |X_A|+|X_B|+|X^*|=|X|\leq c|A\cap X| = c( |X_A|+|X^*|).
    \]
    Adding~${(c-1)|X_B|}$ to both sides, 
    \[
        c|X_B| \leq (c-1)( |X_A|+|X_B|+|X^*|)=(c-1)|X|.
    \]
    Hence,~$|X_B| \leq \frac{c-1}{c}|X|$. 
    By symmetry,~$|X_A| \leq \frac{c-1}{c}|X|$. 
    Hence,~$(A,B)$ is~$\frac{c-1}{c}$-balanced for~$X$.
\end{proof}

Note that it can easily be seen that the above condition is in fact equivalent for a separation to be~$\alpha$-balanced.

The next lemma has appeared in various forms across the literature. 
The first explicit variant, for the values~${c=2}$ and~${q=2k+1}$, was given by \citet{Reed97}, but a similar argument already appears in the graph minors series by \citet{RS-X}.
For our purposes, it is convenient to tie the treewidth of a graph to the existence of~$\alpha$-balanced separations for sets of a \textsl{fixed} size. 

\begin{lem}\label{lemma:balancedSepsForSets}
    Let~${k \geq 1}$, ${c \geq 2}$ and~${q \geq ck+1}$ be integers, and let~$\alpha \coloneqq \frac{c-1}{c}$. 
    If~$G$ is a graph such that for every set~$X\subseteq V(G)$ of size exactly~$q$, $G$ has an~$\alpha$-balanced separation for~$X$ of order at most~$k$, then~${\tw(G) \leq q + k - 1}$.
\end{lem}

\begin{proof}
    We prove the following stronger claim which implies the assertion.
        
        \begin{claim}
            Let ${k,c,q,\alpha}$ and~$G$ be as in the lemma statement. 
            Then, for any set~${X \subseteq V(G)}$ of at most~$q$ vertices, $G$ has a tree-decomposition~${(T,\beta)}$ of width at most~${q+k-1}$, and there is a vertex~${r \in V(T)}$ such that~${X \subseteq \beta(r)}$. 
        \end{claim}
    
    We proceed by induction on~${\lvert V(G)\rvert}$. 
    If~${\lvert V(G)\rvert \leq q+k}$, then we are done, regardless of the set $X$, by defining $T$ to be the tree with vertex set $\{r\}$ and $\beta(r)\coloneqq V(G)$. 
    This settles the base of our induction.
    
    So assume $\lvert V(G)\rvert > q+k$. By adding additional vertices to $X$ if necessary, we may also assume that $\lvert X \rvert = q$. Thus, by assumption, $G$ has an $\alpha$-balanced separation $(A_1,A_2)$ for $X$ such that $|A_1\cap A_2|\leq k$.
    
    For each $i\in\{1,2\}$ let $X_i \coloneqq (X\cap A_i)\cup (A_1\cap A_2)$. Note that $k<\frac{q}{c}$ since $q\geq ck+1$.
    Moreover, since $(A_1,A_2)$ is an $\alpha$-balanced separation for $X$,
    \begin{align*}
        \lvert X_i\rvert
        &\leq \lvert X\cap (A_i \setminus A_{3-i}) \rvert + \lvert A_1 \cap A_2 \rvert \leq \alpha \cdot \lvert X\rvert + k = \tfrac{c-1}{c} \cdot q +k < q.
    \end{align*}

    It follows that $X\cap (A_i\setminus A_{3-i}) = X \setminus X_{3-i}\neq\emptyset$ for both $i\in\{1,2\}$.
    This implies that $G_i\coloneqq G[A_i]$ has fewer vertices than $G$ for both $i\in\{1,2\}$. We now show that each $G_i$ inherits the property from $G$ that for every set $Y \subseteq V(G_i)$ of size exactly $q$, $G_i$ has an $\alpha$-balanced separation for $Y$ of order at most $k$. Indeed, let $Y\subseteq V(G_i)$ with $|Y|=q$. By assumption, $G$ has $\alpha$-balanced separation $(A,B)$ for $Y$ of order at most $k$. Let $A'\coloneqq A\cap V(G_i)$ and $B'\coloneqq B\cap V(G_i)$. So $(A',B')$ is a separation in $G_i$ of order $|A'\cap B'|\leq|A\cap B|\leq k$. Moreover, since $Y\subseteq V(G_i)$, we have $Y\cap (A'\setminus B')=Y\cap (A\setminus B)$ and $Y\cap (B'\setminus A')=Y\cap (B\setminus A)$. So $|Y\cap (A'\setminus B')|\leq\alpha|Y|$ and $|Y\cap (B'\setminus A')|\leq\alpha|Y|$. Thus $(A',B')$ is $\alpha$-balanced for $Y$, as claimed.
    
    Therefore, for each $i \in \{1,2\}$, by applying the induction hypothesis to $G_i$ together with $X_i$ we obtain a tree-decomposition $(T_i,\beta_i)$ of $G_i$ of width at most $q+k-1$ such that there exists a vertex $r_i\in V(T_i)$ with $X_i\subseteq \beta_i(r_i)$.

    We now construct a tree-decomposition $(T,\beta)$ for $G$. 
    Define the tree $T$ to be the disjoint union of $T_1$ and $T_2$ together with a new vertex $r$ adjacent only to $r_1$ and $r_2$. 
     Set $\beta(r)\coloneqq X\cup (A_1\cap A_2)$, and for each $i\in\{1,2\}$ and each $t\in V(T_i)$, 
     set $\beta(t)\coloneqq \beta_i(t)$.
    Thus, $X\subseteq \beta(r)$ and $\lvert \beta(r)\rvert\leq \lvert X\rvert + \lvert A_1\cap A_2\rvert \leq q+k$.
    Also, since $(A_1, A_2)$ is a separation of $G$, for every edge $vw \in E(G)$, there exists $i \in \{1,2\}$ such that $vw \in E(G_i)$, and so $v,w \in \beta_i(t) = \beta(t)$ for some $t \in V(T_i)$. Moreover, note that any $v \in V(G_1) \cap V(G_2)$ appears in $\beta(r_1)$, $\beta(r_2)$, and $\beta(r)$, and any $v \in \beta(r)$ appears in $\beta(r_i)$ for some $i \in \{1,2\}$, so for each $v \in V(G)$, the set $\{t \in V(T) \colon v \in \beta(t)\}$ induces a subtree of $T$.
    Thus $(T,\beta)$ is the desired tree-decomposition for $G$.
\end{proof}

\subsection{Graphs with bounded Hadwiger number}

\citet{DM82} proved the following lower bound on the independence number~$\alpha(G)$ for a graph~$G$ with given Hadwiger number. 

\begin{thm}[\citep{DM82}]
    \label{thm:duchetmeyniel}
    For every graph~$G$, 
    \[
        2 \cdot \alpha(G) \cdot \had(G) \geq \lvert V(G) \rvert.
    \]
\end{thm}

A \defn{path} is a tree whose vertices all have degree at most~$2$. 
A path in a graph~$G$ with endpoints~${v,w \in V(G)}$ is called a \defn{$(v,w)$-path}. 
For~${A,B \subseteq V(G)}$, an \defn{$(A,B)$-path} is a $(v,w)$-path that is internally disjoint from~${A \cup B}$, for some~${v \in A}$ and~${w \in B}$. 
An \defn{$(A,B)$-linkage} is a set of pairwise vertex-disjoint $(A,B)$-paths. 
Two paths~$P$ and~$Q$ in~$G$ are \defn{adjacent} if there is an edge of~$G$ between some~${v \in V(P)}$ and some~${w \in V(Q)}$. 
A linkage is \defn{induced} if its paths are pairwise non-adjacent. 

\begin{cor}
    \label{cor:inducedlinkage}
    For every graph~$G$ and all integers~${k,t \geq 1}$, if~${\had(G) \leq t}$ then every linkage of size at least~${2 k t}$ in~$G$ contains an induced linkage of size at least~$k$. 
\end{cor}

\begin{proof}
    Consider the graph~$H$ obtained from contracting each path in a linkage of size at least~${2 k t }$ in~$G$ to a single vertex and deleting all other vertices. 
    Since~$H$ is a minor of~$G$, ${\had(H) \leq \had(G) \leq t}$ and~${|V(H)| = 2kt}$. 
    By \Cref{thm:duchetmeyniel},~$H$ has an independent set of size at least~$k$, which corresponds to an induced linkage in~$G$ of size at least~$k$. 
\end{proof}

\subsection{Induced minors and vertex-minors of grids}\label{SectionInducedMinorsVertexMinors}

We prove the following theorem, which allows us to find induced minors and vertex-minors of grids. 
For~${m,n \in \NN}$, the \defn{$(m\times n)$-grid} is the graph with vertex set $[m]\times [n]$ where vertices $(v_1,v_2)$ and $(u_1,u_2)$ are adjacent whenever $v_1=u_1$ and $|v_2-u_2|=1$, or $v_2=u_2$ and $|v_1-u_1|=1$.

\begin{thm}\label{InducedMinorVertexMinor}
    Every $n$-vertex planar graph is both an induced minor and a vertex-minor of the $(28n\times 28n)$-grid.
\end{thm}

A \defn{subdivision} of a graph $G$ is any graph $H$ obtained from $G$ by replacing each edge with a possibly longer path. We say that the vertices in $V(H)\setminus V(G)$ are \defn{subdivision vertices}. If all such paths have length greater than $1$, then we say that the subdivision is \defn{proper}, and if all such paths have length exactly two, then we say that $H=G^{(1)}$ is the \defn{1-subdivision} of $G$.

The following two lemmas relate induced minors, vertex-minors, graph minors, and subdivisions.

\begin{lem}\label{inducedminor}
    If a graph $H$ is a minor of a graph $G$, then $H$ is an induced minor of any proper subdivision of $G$.
\end{lem}

\begin{proof}
   Since any proper subdivision of $G$ contains $G^{(1)}$ as an induced minor, it suffices to show that $G^{(1)}$ contains $H$ as an induced minor. Let $G=G_k,G_{k-1}\dots,G_0$ be a sequence of graphs such that $G_{i-1}$ is obtained from $G_{i}$ by either deleting a vertex, deleting an edge, or contracting an edge, and $G_0$ is isomorphic to $H$. Such a sequence exists since $H$ is a minor of $G$. We prove the claim by induction on $k$. If $k=0$, then $G=H$ and so $H$ can be obtained from $G^{(1)}$ by contracting a minimal set of edges $E'$ such that every subdivision vertex is incident to an edge in $E'$. So $H$ is an induced minor of $G^{(1)}$.

    Now suppose $k\geq 1$. There are three cases to consider. If $G_{k-1}$ is obtained from $G_k$ by deleting a vertex $v$, then $G_{k-1}^{(1)}$ can be obtained from $G_k^{(1)}$ by deleting $v$ and all subdivision vertices incident to $v$. If $G_{k-1}$ is obtained from $G_k$ by deleting an edge $e$, then $G_{k-1}^{(1)}$ can be obtained from $G_k^{(1)}$ by deleting the subdivision vertex corresponding to $e$. Finally, if $G_{k-1}$ is obtained from $G_k$ by contracting an edge $e$, then $G_{k-1}^{(1)}$ can be obtained from $G_k^{(1)}$ by contracting both edges incident to the subdivision vertex corresponding to $e$ then deleting the minimal number of degree $2$ vertices required to remove all $4$-cycles. In each case, the claim then follows by induction.    
\end{proof}

\begin{lem}\label{lem:vertexminorminor}
    If a graph $H$ is a minor of a graph $G$, then $H$ is a vertex-minor of any proper subdivision of $G$.
\end{lem}

\begin{proof}
    Let $G^*$ be a proper subdivision of $G$.
    First we shall show that $G^*$ contains $G^{(1)}$ as a vertex-minor. We shall argue this inductively on $|V(G^*)|$. Clearly this holds if $|V(G^*)|=|V(G^{(1)})|$.
    If $|V(G^*)| > |V(G^{(1)})|$, then some edge $uv$ of $G$ is replaced with some path $x_0x_1\ldots x_\ell$ in $G^*$ with $x_0=u$, $x_\ell = v$, and $\ell > 2$. Then observe that $G'=(G^* * x_{\ell -1}) \setminus \{x_{\ell -1}\}$ is a proper subdivision of $G$ with $|V(G')|<|V(G^*)|$.
    Therefore $G'$ and thus $G^*$ contains $G^{(1)}$ as a vertex-minor as desired.
    So, we may now assume that $G^*=G^{(1)}$.

    We argue inductively on $|V(G)|$. For each edge $uv$ of $G$, let $x_{uv}$ be the vertex of $G^{(1)}$ adjacent to both $u$ and $v$.
    If $|V(G)|=|V(H)|$, then we can obtain $H$ as a vertex-minor of $G^{(1)}$ by, for every $uv\in E(H)$, locally complementing at $x_{uv}$, and then, for every $uv \in E(G)$, deleting $x_{uv}$. Next, suppose that $|V(G)|>|V(H)|$. If there is a vertex $v \in V(G)$ such that $G \setminus \{v\}$ contains $H$ as a minor, then we may simply apply the inductive hypothesis to $G \setminus \{v\}$, as clearly $(G \setminus \{v\})^{(1)}$ is a vertex-minor of $G^{(1)}$.
    Otherwise, there exists some edge $uv$ of $G$ such that $M=G/uv$ contains $H$ as a minor.
    Observe that $M^{(1)}$ can be obtained from $(G^{(1)} * u * x_{uv})\setminus \{u, x_{uv}\}$ by deleting the minimal number of degree $2$ vertices to remove all $4$-cycles. 
    By the inductive hypothesis, $M^{(1)}$ contains $H$ as a vertex-minor, so $G^{(1)}$ contains $H$ as a vertex-minor, as desired.
\end{proof}

We need the following lemma of \citet{RST94}.

\begin{lem}[{\protect\citep[(1.5)]{RST94}}]
\label{lem:planarminor}
    Every $n$-vertex planar graph is a minor of the $(14n \times 14n)$-grid.
\end{lem}

\begin{proof}[Proof of \cref{InducedMinorVertexMinor}]
    Let $H$ be an $n$-vertex planar graph. By \cref{lem:planarminor}, $H$ is a minor of the $(14n \times 14n)$-grid. By \cref{inducedminor,lem:vertexminorminor}, $H$ is both an induced minor and a vertex-minor of the $1$-subdivision of the $(14n \times 14n)$-grid. Since the $1$-subdivision of the $(14n\times 14n)$-grid is an induced subgraph of the $(28n\times 28n)$-grid, it follows that $H$ is both an induced minor and a vertex-minor of the $(28n\times 28n)$-grid, as required.
\end{proof}

\section{Induced grid minors}
\label{sec:InducedGridMinors}

This section proves that every graph with large treewidth contains a large induced grid minor or a large complete graph minor. As an intermediate result, we prove \Cref{thm:treewidthplanarminor} with~${f(k,t)\in O(kt)}$. 
The result we actually prove (\Cref{InducedGridMinorOrCompleteMinor} below) is stated in terms of induced grid minors. 
The desired bounds for  \Cref{thm:treewidthplanarminor} can be recovered by combining it with the following folklore result, a proof of which can be extracted from the proof of Theorem 1.1 in~\cite{AAKST21}, or found in the appendix of this paper.

\begin{thm}
\label{AppendixTheorem}
    There is a positive integer~$c$ such that for every positive integer~$t$, if~$G$ contains a ${ct \times ct}$-grid as an induced minor, then~$G$ contains an induced planar subgraph of treewidth at least~$t$. 
\end{thm}

For~${m,n\in \NN}$, let \defn{$P_{m}\boxtimes P_{n}$} denote the graph with vertex set~${[m] \times [n]}$, where vertices $(v_1,v_2)$ and $(u_1,u_2)$ are adjacent whenever~${|v_1-u_1| \leq 1}$ and~${|v_2-u_2| \leq 1}$. 
The \defn{induced grid number~$\igrid(G)$} of a graph~$G$ is the largest integer~$k$ such that~$G$ contains the ${(k\times k)}$-grid as an induced minor.

Our main result in this section is the following. 

\begin{thm}
\label{InducedGridMinorOrCompleteMinor}
    For every~${k,t \in \NN}$, every graph~$G$ that contains the ${(t(2k+1) \times t(2k+1))}$-grid as a minor contains the ${(k \times k)}$-grid as an induced minor or~$K_t$ as a minor.
\end{thm}

\begin{proof}
    By considering an appropriate induced minor of~$G$, we may assume that ${V(G) = [t(2k+1)] \times [t(2k+1)]}$, where vertices~${(v_1,v_2)}$ and~${(u_1,u_2)}$ are adjacent whenever~${v_1 = u_1}$ and~${|v_2-u_2|=1}$, or~${v_2 = u_2}$ and~${|v_1-u_1| = 1}$. 
    (Note that~$G$ may contain other edges as well.) 
    An edge~${(v_1,v_2)(u_1,u_2) \in E(G)}$ is a \defn{jump} if~${|v_1-u_1| \geq 2}$ or~${|v_2-u_2| \geq 2}$. 
    
    As illustrated in \cref{fig:Gab}, for~${a \in [t]}$ and~${b \in [t]}$, 
    let~$G_{a,b}$ be the subgraph of~$G$ induced by 
    \[
        [(a-1)(2k+1)+1,a(2k+1)] \times [(b-1)(2k+1)+1,b(2k+1)]
    \]
    and 
    let $H_{a,b}$ be the subgraph of $G$ induced by 
    \[
        [(a-1)(2k+1)+2,a(2k+1)-1] \times [(b-1)(2k+1)+2,b(2k+1)-1]
    \]
    Note that~$G_{a,b}$ contains a ${((2k+1) \times (2k+1))}$-grid, and~$H_{a,b}$ contains a ${((2k-1) \times (2k-1))}$-grid strictly inside~$G_{a,b}$. 
    Also note that~$G_{a,b}$ and~$G_{c,d}$ are disjoint for~${(a,b) \neq (c,d)}$. 
    We refer to the bottom-left, bottom-right, top-left and top-right corners of~$G_{a,b}$ as illustrated in \cref{fig:Gab}. 
    
    \begin{figure}[ht]
        \centering
        \includegraphics{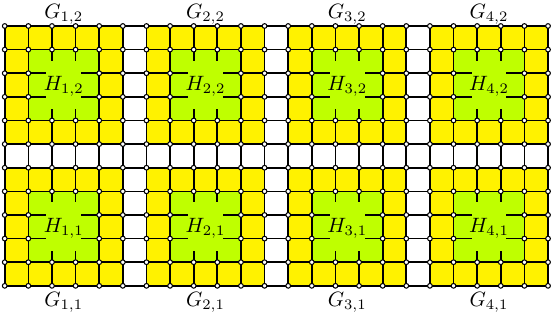}
        \caption{Partition into subgraphs~$G_{a,b}$.}
        \label{fig:Gab}
    \end{figure}
    
    Suppose there exists~${a \in [t]}$ and~${b \in [t]}$ such that~$H_{a,b}$ contains no jump edge. 
    Thus~$H_{a,b}$ is isomorphic to a subgraph of~${P_{2k-1} \boxtimes P_{2k-1}}$. 
    We claim that~$H_{a,b}$ contains the ${(k \times k)}$-grid as an induced minor. 
    For ease of notation, assume without loss of generality that~${a=b=1}$. 
    For each~${i,j \in [k]}$, let 
    \begin{align*}
        X_{i,j} &\coloneqq G[\{(2i,2j-1),(2i,2j),(2i,2j+1)\}\cap V(H_{1,1})] \textrm{\phantom{l} if $i+j$ is even, and}\\
        X_{i,j} &\coloneqq G[\{(2i-1,2j),(2i,2j),(2i+1,2j)\}\cap V(H_{1,1})] \textrm{\phantom{l} if $i+j$ is odd.}
    \end{align*}
    Observe that~$X_{i,j}$ is connected, and that~$X_{i,j}$ and~$X_{i',j'}$ are adjacent if and only if~${i=i'}$ and~${|j-j'|=1}$, or~${j=j'}$ and~${|i-i'|=1}$ (see \cref{fig:NoJump} where each~$X_{i,j}$ is a path). 
    Thus, the induced minor obtained by contracting each~$X_{i,j}$ into a vertex and then deleting the remaining vertices is the ${(k\times k)}$-grid. 
    
    \begin{figure}[ht]
        \centering
        \includegraphics[angle=90]{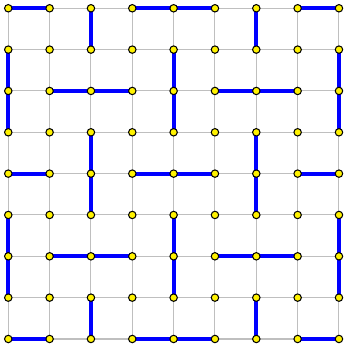}
        \caption{Finding an induced $(k\times k)$-grid when there are no jump edges in $H_{a,b}$. 
        }
        \label{fig:NoJump}
    \end{figure}
    
    Now assume that~$H_{a,b}$ contains a jump edge for each~${a \in [2t]}$ and~${b \in [t]}$. 
    We now work towards showing that~$G$ contains a $K_t$-minor. 
    
    \begin{claim}
        For all~${a \in [t]}$ and~${b \in [t]}$, there exist vertex-disjoint paths~$P_{a,b}$ and~$Q_{a,b}$ in~$G_{a,b}$ such that the endpoints of~$P_{a,b}$ are the bottom-left and top-right corners of~$G_{a,b}$, and the endpoints of~$Q_{a,b}$ are the top-left and bottom-right corners of~$G_{a,b}$. 
    \end{claim}
    
    \begin{proof}
        By assumption, the graph~$H_{a,b}$ contains a jump edge  ${(x_1,y_1)(x_2,y_2) \in E(G)}$. 
        We only discuss the case where~${x_2 \geq x_1+2}$, since all other cases are analogous. 
        Choose~${x \in \NN}$ so that~${x_1 < x < x_2}$. 
        For ease of notation, assume without loss of generality that~${a=b=1}$. 
        As illustrated in \cref{fig:RootedPaths}, define the paths  
        \[
            P_{1,1}\coloneqq( 
            (1,1),
            \dots,
            (1,y_1)
            ,\dots,
            (x_1,y_1),(x_2,y_2),
            \dots,
            (2k+1,y_2),
            \dots,
            (2k+1,2k+1))
        \]
        and 
        \[
            Q_{1,1}\coloneqq(
            (1,2k+1),
            \dots,
            (x,2k+1),
            \dots,
            (x,1),
            \dots,
            (2k+1,1)).
        \]
        By construction, $P_{1,1}$ joins the bottom-left and top-right corners of~$G_{1,1}$, while~$Q_{1,1}$ joins the top-left and bottom-right corners of~$G_{1,1}$. 
        Moreover, $P_{1,1}$ and~$Q_{1,1}$ are disjoint since every vertex in~$Q_{1,1}$ has an x-coordinate of~$x$ or a y-coordinate of~$1$ or~$2k+2$, while no internal vertex of~$P_{1,1}$ has an x-coordinate of~$x$ or a y-coordinate of~$1$ or~$2k+2$. 
    \end{proof}
    
    \begin{figure}[ht]
        \centering
        \includegraphics{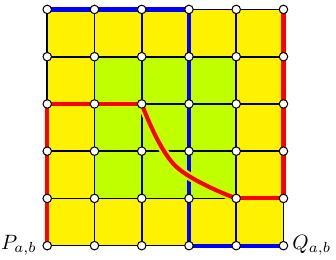}
        \caption{Finding vertex-disjoint rooted paths $P_{a,b}$ and $Q_{a,b}$.}
        \label{fig:RootedPaths}
    \end{figure}
    
    We now show that~$K_t$ is a minor of~$G$. 
    For each~${a \in [t]}$ and~${b \in [t]}$, contract~$P_{a,b}$ to an edge joining the bottom-left corner and top-right corner of~$G_{a,b}$, contract~$Q_{a,b}$ to an edge joining the top-left corner and bottom-right corner of~$G_{a,b}$, then contract each nontrivial path of~${G_{a,b}-E(P_{a,b})-E(Q_{a,b})-V(H_{a,b})}$ down to a single edge. 
    These operations contract~$G_{a,b}$ to a copy of~$K_4$ with vertex set the corners of~$G_{a,b}$. 
    Since~$G_{a,b}$ and~$G_{c,d}$ are disjoint for~${(a,b) \neq (c,d)}$, resulting copies of~$K_4$ are pairwise disjoint. 
    For each~${a \in [t-1]}$ and~${b \in [t]}$, contract the grid edges between~$G_{a,b}$ and~$G_{a+1,b}$. 
    For each~${a \in [t]}$ and~${b \in [t-1]}$, contract the grid edges between~$G_{a,b}$ and~$G_{a,b+1}$. 
    This produces~${P_{t+1} \boxtimes P_{t+1}}$.

    It is folklore that~$K_s$ is a minor of~${P_{2s} \boxtimes P_{2s}}$ for any integer~${s \geq 1}$ (see~\citep{Tarik} for example). 
    We now show that~$K_{2s}$ is a minor of~${P_{2s} \boxtimes P_{2s}}$. 
    Applying this to whichever of~$t$ o~ $t+1$ is even, will show that~$K_t$ is a minor of~$G$, as desired. 
    As illustrated in \cref{GridWithCrossesEven}, for each~${i \in [s]}$ consider the sets 
    \begin{align*}
        Q_i &\coloneqq \{(x,2i-x) \colon x\in[2i-1]\}\cup\{ (x,x-2i+1)  \colon x\in[2i,2s] \} \text{ and}\\
        R_i &\coloneqq \{(x,x+2s+1-2i) \colon x\in[2i-1]\} \cup \{(x,2s+2i-x):x\in[2i,2s]\}.
    \end{align*}
    Observe that each of~${Q_1,\dots,Q_s,R_1,\dots,R_s}$ induces a path, 
    and that ${Q_1,\dots,Q_s,R_1,\dots,R_s}$ are pairwise disjoint and pairwise adjacent. 
    \begin{figure}[!h]
       \centering
        \includegraphics[scale=0.8]{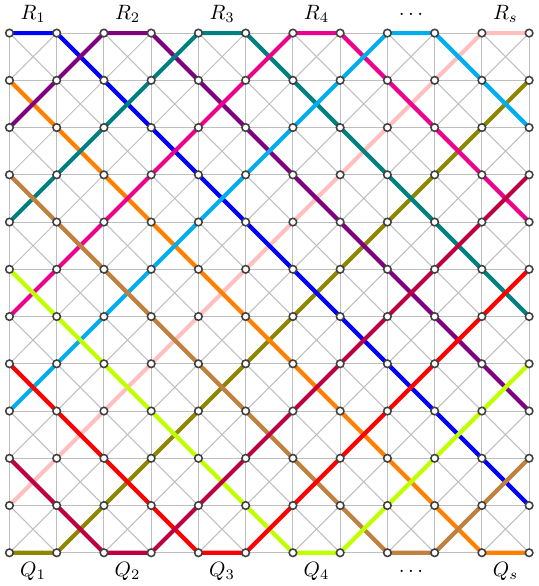}
        \caption{$K_{2s}$ minor in $P_{2s}  \boxtimes P_{2s}$.}
        \label{GridWithCrossesEven}
    \end{figure}
\end{proof}

We now prove the following result from \cref{sec:intro}.

\NewMainEquivalent*
\begin{proof}
    Clearly \ref{(2)} implies \ref{(3)} by \Cref{thm:polygrid}, and~\ref{(3)} implies~\ref{(1)}. 
    To show that \ref{(4)} implies \ref{(2)}, suppose~$\GG$ excludes an $n$-vertex planar graph~$H$ as an induced minor. 
    \cref{InducedMinorVertexMinor} implies that~${\igrid(G) < 28n}$ for every~${G \in \GG}$, and so~\ref{(2)} follows from \cref{InducedGridMinorOrCompleteMinor} and the definition of~$g_{\mathcal G}$. 
    It remains to show that \ref{(1)} implies \ref{(4)}. 
    Suppose~$\GG$ does not exclude a planar graph as an induced minor. 
    So for every~${k \in \NN}$, there is a graph~${G_k' \in \GG}$ such that~$G_k'$ contains the $(k\times k)$-grid as an induced minor. 
    By \cref{AppendixTheorem}, $\GG$ contains planar graphs of arbitrarily large treewidth, which contradicts that $\GG$ is $(\tw,\had)$-bounded (since planar graphs have Hadwiger number at most 4). 
\end{proof}

Note that excluding a planar graph $H$ as a minor is sufficient to exclude $H$ as an induced minor and also implies a bound of $|V(H)|$ on the Hadwiger number.
Thus we can recover the Grid Minor Theorem from \Cref{MainEquivalent1} by observing that the class of $H$-minor-free graph must be a $(\tw,\had)$-bounded class with bounded Hadwiger number, and thus must have bounded treewidth.
However since our proof uses the Grid Minor Theorem, we do not obtain a new proof of it.

We conclude this section with the following open problem.  
\cref{InducedGridMinorOrCompleteMinor} roughly says that every graph~$G$ that contains the ${(6kt \times 2kt)}$-grid as a minor contains the ${(k\times k)}$-grid as an induced minor or a $K_t$-minor. 
It is natural to ask whether the~$kt$ term here is best possible (up to constant factors). 
The ${((k-1) \times (k-1))}$-grid contains no ${(k \times k)}$-grid as an induced minor, and contains no $K_5$-minor. 
So the lemma is best possible for fixed~${t \geq 5}$. 
At the other extreme, ${G \coloneqq P^{t}_{t^2}}$ has treewidth~$t$ and thus contains no $K_{t+2}$-minor. 
Moreover, $G$ is chordal, so~$G$ contains no ${2 \times 2}$-grid as an induced minor. 
Finally, $G$ contains a ${t \times t}$-grid as a spanning subgraph. 
So for fixed~${k \geq 2}$, \cref{InducedGridMinorOrCompleteMinor} is best possible. 
But what if neither~$k$ nor~$t$ is fixed? 
We cannot exclude the possibility that the~$kt$ terms in \cref{InducedGridMinorOrCompleteMinor} can be replaced by ${c(k+t)}$.

\section{Vertex-minors}\label{SectionVertexMinor}

In this section, we show that every proper vertex-minor-closed class is polynomially $(\tw,\had)$-bounded.

\begin{thm}\label{MainVertexMinor}
    Every proper vertex-minor-closed class is polynomially $(\tw,\had)$-bounded with $(\tw,\had)$-bounding function in $O(\had(G)^9 \polylog(\had(G))$.
\end{thm}

By \cref{MainEquivalent1}, it suffices to show that for every proper vertex-minor-closed class $\GG$, there is a planar graph $H$ such that every graph in $\GG$ excludes $H$ as an induced minor.

We first show that every graph is a vertex-minor of a sufficiently large grid. This was also shown by \citet{RB2001} in the language of quantum computing. The following unpublished proof, originally due to Sang-il Oum, uses a combinatorial approach.  

The \defn{crossing number} of a graph $G$ is the minimum number of crossings in a drawing of $G$ in the plane, where no three edges cross at the same point. See \citep{Schaefer22} for a formal definition and background on crossing number. 

\begin{thm}\label{thm:vertexminorgrids}
    Every $n$-vertex $m$-edge graph $G$ with crossing number $k$ is a vertex-minor of the $(28(n+m+3k)\times 28(n+m+3k))$-grid.
\end{thm}

\begin{proof}
    Let~$D_0$ be a drawing of~$G$ in the plane with~$k$ crossings.
    Let~$D_1$ be the planar graph obtained from~$D_0$ by replacing each crossing point~$c$ by a vertex of degree 4. 
    We call such vertices \defn{crossing vertices}. 
    So~${|V(D_1)|=n+k}$ and~${|E(D_1)|=m+2k}$. 
    Let~$D_2$ be the $1$-subdivision of~$D_1$. 
    Then~$D_2$ is planar and~${|V(D_2)|=n+m+3k}$. 
    For each crossing vertex~$c$ of~$D_2$, let its neighbours in clockwise-wise order be~$c_1,c_2,c_3,c_4$.
    Let~$D_3$ be the planar graph obtained from~$D_2$ by adding, for each crossing vertex~$c$ of~$D_2$,  the edges~${c_1c_2,c_2c_3,c_3c_4,c_4c_1}$. 
    So~${|V(D_3)|=|V(D_2)|=n+m+3k}$. 
    By \cref{InducedMinorVertexMinor}, $D_3$ is a vertex-minor of the ${(28(n+m+3k)\times 28(n+m+3k))}$-grid.

    Let~$G'$ be the graph obtained from~$D_3$ by, for each crossing vertex~$c$ of~$D_3$, locally complementing at~$c$, and then deleting~$c$. 
    Then~$G'$ is a vertex-minor of~$D_3$ and equivalently, $G'$ is obtained from~$D_3$ by, for each crossing vertex~$c$ of~$D_3$, removing~$c$ and the edges~${c_1c_2,c_2c_3,c_3c_4,c_4c_1}$, and then adding the edges~$c_1c_3$ and~$c_2c_4$.
    Equivalently, $G'$ is obtained from~$D_2$ by, for each crossing vertex~$c$, deleting the vertex~$c$ and adding the edges~$c_1c_3$ and~$c_2c_4$.
    So, $G'$ is a proper subdivision of $G$.
    Therefore, by \cref{lem:vertexminorminor}, $G'$ contains~$G$ as a vertex-minor.
    Hence, the $(28(n+m+3k)\times 28(n+m+3k))$-grid contains~$G$ as a vertex-minor.
\end{proof}

Next, we show that large grids are vertex-minors of graphs that contain huge grids as an induced minor.
We need the following useful lemma on vertex-minors.

\begin{lem}\label{lem:vertexminormaxdegree}
    Let $H$ be any graph with maximum degree 3. Let $G$ be any graph that contains the $3$-subdivision of $H$ as an induced minor. Then $G$ contains a proper subdivision of $H$ as a vertex-minor.
\end{lem}

\begin{proof}
    By possibly deleting vertices, we may assume that there is a partition ${\bigcup_{v\in V(H)} Y_v \cup \bigcup_{uv\in E(H)} (Y_{uv}^1\cup Y_{uv}^2 \cup Y_{uv}^3)}$ of~$V(G)$ such that each of these sets of vertices induces a connected subgraph and the graph obtained from~$G$ by contracting the edges of the induced subgraph of each of these sets is the $3$-subdivision of $H$, with vertices naturally associated.
    By possibly deleting more vertices, we may assume that for each~${uv \in E(H)}$, ${G[Y_{uv}^1 \cup Y_{uv}^2 \cup Y_{uv}^3]}$ is a path with at least three vertices whose internal vertices have no neighbour in~${V(G)\setminus (Y_{uv}^1 \cup Y_{uv}^2 \cup Y_{uv}^3)}$.
    For each~${uv \in E(H)}$, let~$a_{uv}$ be an internal vertex of the path~${G[Y_{uv}^1 \cup Y_{uv}^2 \cup Y_{uv}^3]}$.
    Notice that the vertices of $\{a_{uv} \colon uv\in E(H)\}$ are pairwise at distance at least~$4$ in~$G$.
    Then for each vertex~$u$ of~$H$ with neighbours~${v_1, \ldots, v_r}$ in~$H$ (where~${r \leq 3}$ since~$H$ has maximum degree~$3$), there is a set~$X_u$ of vertices with~${Y_u \subseteq X_u}$ such that~$G[X_u]$ is a connected component of~${G \setminus \{a_{uv_1}, \ldots , a_{uv_r}\}}$ and~${G[X_u \cup \{a_{uv_1}, \ldots , a_{uv_r}\}]}$ is connected and each vertex ${a_{uv_1}, \ldots , a_{uv_r}}$ has degree~$1$ in~${G[X_u \cup \{a_{uv_1}, \ldots , a_{uv_r}\}]}$. 
    By possibly deleting more vertices, we may further assume that for each vertex $u$ of $H$, $G[X_u \cup \{a_{uv_1}, \ldots , a_{uv_r}\}]$ is either a path or a subdivided~$K_{1,3}$ with $a_{uv_1}, \ldots , a_{uv_r}$ being leaves, or the line graph of a subdivided $K_{1,3}$ with leaves $a_{uv_1}, a_{uv_2}, a_{uv_3}$. 
    Next, we shall reduce the line graph case to the non-line graph case. 
    
    Suppose that $G[X_u \cup \{a_{uv_1}, a_{uv_2}, a_{uv_3}\}]$ is the line graph of a subdivided $K_{1,3}$ with leaves $a_{uv_1}, \ldots , a_{uv_3}$, since $a_{uv_1}, a_{uv_2}, a_{uv_3}$ are pairwise at distance at least 4 in $G$, there is a vertex $t$ of the triangle in $G[X_u \cup \{a_{uv_1}, a_{uv_2}, a_{uv_3}\}]$ such that $t\not\in \{a_{uv_1}, a_{uv_2}, a_{uv_3}\}$, and $t$ has a neighbour $x$ not contained in the triangle, with $x\not\in \{a_{uv_1}, a_{uv_2}, a_{uv_3}\}$.
    Then $(G[X_u \cup \{a_{uv_1}, a_{uv_2} , a_{uv_3}\}]*t)\setminus \{t\}$ is isomorphic to the graph obtained from $G[X_u \cup \{a_{uv_1}, a_{uv_2} , a_{uv_3}\}]$ by deleting the edge of the triangle not incident to $t$, and then contracting $tx$.
    So $(G[X_u \cup \{a_{uv_1}, a_{uv_2} , a_{uv_3}\}]*t)\setminus \{t\}$ is a subdivided $K_{1,3}$ tree with leaves $a_{uv_1}, a_{uv_2}, a_{uv_3}$.
    Since the local complementation did not alter any edges of~$G$ outside $G[X_u \cup \{a_{uv_1}, a_{uv_2} , a_{uv_3}\}]$, by doing this for each $u\in V(H)$, it now follows that~$G$ contains a proper subdivision of~$H$ as a vertex-minor.
\end{proof}

\begin{lem}\label{lem:vertminorgrids}
    Every graph $G$ with~${\igrid(G) \geq 308 k^2}$ contains the ${(k \times k)}$-grid as a vertex-minor.
\end{lem}

\begin{proof}
    By replacing degree-$4$ vertices in the $(k\times k)$-grid with two adjacent degree-$3$ vertices, observe that there is a planar graph~$H$ with maximum degree~$3$ on at most~$2k^2$ vertices that contains the ${(k \times k)}$-grid as a minor.
    The $3$-subdivision of~$H$ has at most~$11k^2$ vertices.
    By \Cref{InducedMinorVertexMinor}, the $(308k^2 \times 308k^2)$-grid contains the $3$-subdivision of~$H$ as an induced minor.
    Therefore, $G$ contains the $3$-subdivision of~$H$ as an induced minor.
    Hence by \cref{lem:vertexminormaxdegree}, $G$ contains a proper subdivision of~$H$ as a vertex-minor.
    Since~$H$ contains the ${(k \times k)}$-grid as a minor, it now follows from \cref{lem:vertexminorminor} that~$G$ contains the ${(k \times k)}$-grid as a vertex-minor.
\end{proof}

From \cref{thm:vertexminorgrids} and \cref{lem:vertminorgrids}, we obtain the following.

\begin{thm}\label{VertexMinorGrid}
    For every proper vertex-minor-closed class $\GG$, there exists $k\in \NN$ such that $\igrid(G)\leq k$ for all $G\in \GG$.
\end{thm}

Together with \cref{MainEquivalent1}, \cref{VertexMinorGrid} implies \cref{MainVertexMinor}.

\section{Ordered graphs and outer-string graphs}\label{sec:orderedgraphs}

In this section, we study ordered graphs that are linearly $(\tw,\had)$-bounded. An \defn{ordered graph ${(G,\leq)}$} is a graph $G$ equipped with a total ordering $\leq$ on its vertex set. Ordered graphs are useful in converting geometric and topological problems into combinatorial problems. To date, most of the work on ordered graphs addresses extremal questions~\citep{GMNPTV23,KT23,Tardos19}. Here, we take a more structural approach.

Let~${(G,\leq)}$ be an ordered graph. 
Let~${uv,xy \in E(G)}$ be a pair of edges with distinct endpoints where~${u < v}$ and~${x < y}$. 
We say that~$uv$ and~$xy$ \defn{cross} if~${u < x < v < y}$ or~${x < u < y < v}$. 
For two sets~${X,Y \subseteq V(G)}$, we write~${X < Y}$ if~${x < y}$ for every~${x \in X}$ and every~${y \in Y}$. An ordered graph $(G,\leq)$ is \defn{$\times$-free} if for every pair~${uv, xy}$ of crossing edges, we have~${\{ux,uy,vx,vy\} \cap E(G) \neq \emptyset}$. A graph $G$ \defn{admits a $\times$-free ordering} if there is a total ordering $\leq$ of $G$ such that $(G,\leq)$ is $\times$-free.

Our main result in this section shows that the class of graphs admitting a $\times$-free ordering is linearly $(\tw,\had)$-bounded.

\begin{thm}\label{thm:orderedgraphs}
    For every $\times$-free ordered graph $(G,\leq)$,
    \[ 
        \tw(G) \leq 15 \had(G) - 2. 
    \]
\end{thm}

We now work towards proving \cref{thm:orderedgraphs}. 
The next lemma is the heart of the proof. 
For this, we recall the classical result of Menger which states that for any graph~$G$ and any two sets~${S, T \subseteq V(G)}$, the maximum cardinality of an $(S,T)$-linkage is equal to the minimum order of a separation~$(A,B)$ of~$G$ with~${S \subseteq A}$ and~${T \subseteq B}$.
 
\begin{lem}
    \label{lemma:orderedbalancedsep}
For every $\times$-free ordered graph~${(G,\leq)}$ with no $K_{t+1}$-minor, for every set~$X$ of exactly~${12t}$ vertices in $G$, $G$ has a ${\frac{3}{4}}$-balanced separation for $X$ of order less than~${3t}$.
\end{lem}

\begin{proof}
    Note that the theorem holds trivially for~${t \in \{0,1\}}$, so we may assume that~${t \geq 2}$. 
    
    Let~$\alpha \coloneqq \frac{3}{4}$. 
    Let~${X \subseteq V(G)}$ be a set of $12t$ vertices.
    Let~$X_1$, $X_2$, $X_3$, $X_4$ be pairwise disjoint sets of~$3t$ vertices in~$X$ with~${X_1 < X_2 < X_3 < X_4}$. 
    Let~$\{ I_1, I_2, I_3, I_4 \}$ be a partition of~$V(G)$ into intervals such that~${X_i \subseteq I_i}$ for all~${i \in [4]}$, and note that~${I_1 < I_2 < I_3 < I_4}$.
    Suppose~$G$ does not contain an $(I_1,I_3)$-linkage of size~${3t}$. 
    By Menger's Theorem, there exists a separation~${(A,B)}$ of~$G$ with order less than~${3t}$ such that~${X_1 \subseteq I_1 \subseteq A}$ and~${X_3 \subseteq I_3 \subseteq B}$. By \Cref{obs:balancedseps}, this separation is~$\alpha$-balanced for~$X$. 
    Likewise, if~$G$ does not contain an $(I_2,I_4)$-linkage of size~$3t$, then~$G$ contains a ${\alpha}$-balanced separation for~$X$ of order less than~${3t}$. 

    So we may assume that $G$ contains an $(I_1,I_3)$-linkage $\mathcal{P}$ of size~${3t}$. Let~$\mathcal{P} = \{P_1,\dots, P_{3t}\}$ ordered by their leftmost end-vertex. 
    Consider for~${i < j}$ the paths~${P_i = (u_1, \dots, u_{m})}$ and~${P_j =(  w_1, \dots, w_{\ell})}$ where~${u_1, w_1 \in I_1}$, and~${u_m, w_{\ell} \in I_3}$.
    Now suppose that~${u_2 \in I_2}$. 
    Since~${u_1 \leq w_1 \leq u_2 \leq w_{\ell}}$, it follows that there is an edge in~$E(P_j)$ that crosses~${u_1u_2}$. 
    Thus~$P_i$ and~$P_j$ are adjacent since~${(G,\leq)}$ is $\times$-free. 
    As such, if at least~${t+1}$ of the paths in~$\mathcal{P}$ have their second vertex in~$I_2$, then~$G$ contains a $K_{t+1}$-minor, a contradiction. 
    Thus, at most~$t$ of the paths in~$\mathcal{P}$ have their second vertex in~$I_2$. 
    Now suppose that~${w_2 \in I_4}$. 
    Similar to before, since~${u_1 \leq w_1 \leq u_m \leq w_{2}}$ and~${(G,\leq)}$ is $\times$-free, it follows that~$P_i$ and~$P_j$ are adjacent. 
    Thus, at most~$t$ of the paths in~$\mathcal{P}$ have their second vertex in~$I_4$. 
    As such, there exist~${t}$ pairwise vertex-disjoint $(I_1,I_3)$-paths~${P_1',\dots P_{t-1}'}$ of length~$1$. 
    
    By a symmetric argument to the above, we may assume that there exist~$t$ pairwise vertex-disjoint $(I_2,I_4)$-paths~${Q_1',\dots Q_{t-1}'}$ of length~$1$. 
    Since each path is of length~$1$, it follows that~${P_1',\dots P_{t-1}',Q_1',\dots Q_{t-1}'}$ are pairwise vertex-disjoint. 
    Moreover, since~${I_1 < I_2 < I_3 < I_4}$, the edge of~$P_i'$ crosses the edge of~$Q_j'$ for all~${i,j \in [t]}$. 
    Since~$G$ is $\times$-free, $P_i'$ and $Q_j'$ are adjacent. 
    Thus by contracting each~$P_i'$ and~$Q_j'$ into a vertex, we obtain a $K_{t,t}$-minor, and hence a~$K_{t+1}$-minor, a contradiction. 
\end{proof}

\begin{proof}[Proof of \cref{thm:orderedgraphs}]
    Let~${t \coloneqq \had(G)}$. 
    By \Cref{lemma:orderedbalancedsep},         for every set~${X \subseteq V(G)}$ of size~$12t$, $G$ has a $\frac34$-balanced separation for $X$ of order at most~$3t-1$. Hence, by \Cref{lemma:balancedSepsForSets} with~${k = 3t-1}$, ~${c=4}$, and~${q = 12t}$,
    we have~${\tw(G) \leq 12t + 3t - 2 = 15t - 2}$. 
\end{proof}

\citet{DKMW23} previously studied $\times$-free ordered graphs in the context of $\chi$-bounded classes. They stated without proof 
that every outer-string graph admits a $\times$-free ordering. We include the proof for the sake of completion.

Recall that an \defn{outer-string graph} is an intersection graph of a set $\mathcal{S}$ of curves inside a closed disk $\Delta$ where for each curve~${S \in \mathcal{S}}$, there is an endpoint of~$S$, called its \defn{root}, on the boundary~$\boundary(\Delta)$ of the disk~$\Delta$. 
We denote the root of~$S$ by \defn{$\Root(S)$}. 
We assume, without loss of generality, that~${\Root(S) \neq \Root(S')}$ and~${S \cap \boundary(\Delta) = \{\Root(S)\}}$ for distinct~${S,S' \in \mathcal{S}}$.

\begin{lem}[\cite{DKMW23}]
\label{OuterStringOrdered}
    Every outer-string graph~$G$ admits a $\times$-free ordering. 
\end{lem}

\begin{proof}
    Let~$\Delta$ be a disk with boundary~$\boundary(\Delta)$. 
    Let $\mathcal{S}$ be a set of curves inside~$\Delta$ where each curve has a root on~$\boundary(\Delta)$ such that~$G$ is the intersection graph of~$\mathcal{S}$. 
    Let~$\leq$ be any linear order on~$V(G)$ obtained from the clockwise cyclic order given by the roots of~$\mathcal{S}$ on~$\Delta$. 
    We claim that with this order, $G$ is $\times$-free. 
    Indeed, if~${xy, uw \in E(G)}$ with~${x < u < y < w}$, then by the Jordan Curve Theorem, ${\boundary(\Delta) \cup \{ x, y \}}$ contains a simple closed curve~$C$ such that one region of~${\Delta \setminus C}$ contains a point of~${u \setminus \{ \Root(u) \}}$, and the other region contains a point of~${w \setminus \{\Root(w)\}}$. Since~$uw \in E(G)$, these two points are joined by a curve~$C'$ in~${(u \cup w) \setminus \{ \Root(u), \Root(w) \}}$.
    Then $C'$ has nonempty intersection with~$C$, and since~${(u \cup w) \setminus \{ \Root(u), \Root(w) \}}$ is disjoint from~$\boundary(\Delta)$, at least one of~$u$ or~$w$ intersects one of~$x$ or~$y$, as desired. 
\end{proof}

\cref{thm:orderedgraphs,OuterStringOrdered} imply \cref{thm:outerstring}.

We conclude this section with the following open problem.
\begin{openproblem}
    For which ordered patterns $X$ is the class of $X$-free ordered graphs linearly $(\tw,\had)$-bounded?
\end{openproblem}

\section{Outer-string graphs on surfaces}
\label{sec:outerstring}

In this section we prove \cref{thm:outerstringgeneral} which shows that outer-string graphs on surfaces of bounded Euler genus are linearly $(\tw,\had)$-bounded. We in fact prove a generalisation that allows for more than one boundary component where strings can be rooted. To formalise this, we need the following definitions. 

Let~$\Sigma_0$ be a surface\footnote{A surface is a two-dimensional manifold with (possibly empty) boundary.} without boundary and let~${c \geq 1}$ be some integer. 
We say that a set~${\Sigma \subseteq \Sigma_0}$ is a \defn{surface with~$c$ cuffs} if~$\Sigma$ can be obtained from~$\Sigma_0$ by deleting the interiors of~$c$ pairwise disjoint closed disks~${\Delta_1,\dots,\Delta_c}$. 
The disks~$\Delta_i$ are called the \defn{cuffs} of~$\Sigma$ and the \defn{Euler-genus} of~$\Sigma$ is the Euler-genus of~$\Sigma_0$. 
Note that a graph~$G$ has an embedding on~$\Sigma$ if and only if it has an embedding on~$\Sigma_0$.

Let~${c \geq 1}$ and~${g \geq 0}$ be integers. 
An \defn{outer-string diagram} of \defn{Euler-genus $g$} with \defn{$c$ cuffs} is a tuple $(\mathcal{S},\Sigma)$ where $\mathcal{S}$ is a family of strings drawn in a surface $\Sigma$ of Euler-genus $g$ with $c$ cuffs such that for each $S\in\mathcal{S}$ there exists a unique cuff $\Delta_S$ where $S\cap \Delta_S\neq\emptyset$ and $S$ intersects $\Delta_S$ exactly in one of its endpoints.
We call the unique point $S\cap \Delta_S$ the \defn{root} of $S$ and denote it by $\Root(S)$.

A graph $G$ is an \defn{outer-string graph} of \defn{Euler-genus $g$} with \defn{$c$ cuffs}, or a \defn{$(g,c)$-outer-string graph}, if there exists an outer-string diagram $(\mathcal{S},\Sigma)$ of Euler-genus $g$ with $c$ cuffs such that $G$ is the intersection graph\footnote{That is, $V(G)=\mathcal{S}$ and $st\in E(G)$ if and only if $s\cap t\neq\emptyset$.} of $\mathcal{S}$.
We say that $(\mathcal{S},\Sigma)$ is a \defn{$(g,c)$-outer-string diagram} for $G$.

We first prove a lemma to help us embed a $K_{3,d}$ into a surface.

For a surface $\Sigma$ with cuffs, a \defn{$\Sigma$-rooted tripod} is the image~$T$ of an embedding of~$K_{1,3}$ into~$\Sigma$ such that the intersection of the boundary of~$\Sigma$ and~$T$ are the leafs of~$K_{1,3}$. 
The point in the boundary of~$\Sigma$ where a leaf is embedded is a \defn{root} of~$T$ and the point where the vertex of degree~$3$ is embedded is the \defn{centre} of~$T$. 

\begin{lem}\label{lemma:embeddingK3n}
    Let~${g \geq 0}$, ${c \geq 1}$, and~${n \geq 3}$ be integers and let~$\Sigma$ be a surface of Euler-genus~$g$ with~$c$ cuffs. Assume there is a set~$\mathcal{T}$ of~$n$ pairwise disjoint $\Sigma$-rooted tripods in~$\Sigma$ such that there are, not necessarily distinct, cuffs~$\Delta$, $\Delta'$ of~$\Sigma$ and two disjoint segments~$A, B$ of the boundary~$\boundary(\Delta)$ of~$\Delta$, such that each tripod has exactly one of its roots in each of the sets~$A$, $B$, and~$\boundary(\Delta') \setminus (A \cup B)$. 
    Then there is a drawing without crossings of~$K_{3,n}$ in~$\Sigma \cup \Delta \cup \Delta'$.
\end{lem}

\begin{proof}
    Let~$R$ denote the set of all the roots of all the tripods in~$\mathcal{T}$. 
    In~${\Delta \cup \Delta'}$ we embed the disjoint union of three pairwise disjoint copies of~$K_{1,n}$, denoting their image by~$S_A$, $S_B$, and $S_{\boundary(\Delta') \setminus (A \cup B)}$, such that for each~${X \in \{A,B,\boundary(\Delta') \setminus (A \cup B)\}}$, the leafs of the copy mapped to~$S_X$ are mapped to~${R \cap X}$. 
    Note that~$S_A$ and~$S_B$ are in~$\Delta$ and $S_{\boundary(\Sigma) \setminus (A \cup B)}$ is embedded into~$\Delta'$.
   
    Now~$\bigcup \mathcal{T} \cup S_A \cup S_B \cup S_{\boundary(\Sigma) \setminus (A \cup B)}$ is the image of an embedding~$K_{3,n}$, as desired, see \Cref{fig:K33onthetorusnew}. 
\end{proof}

\begin{figure}[ht]
    \centering
    \includegraphics{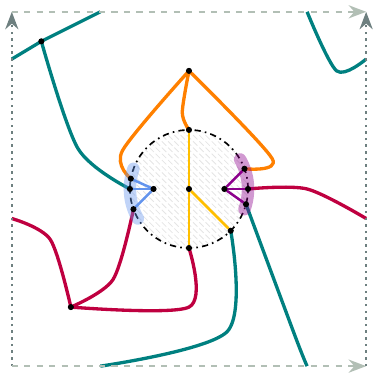}
    \caption{Embedding a $K_{3,3}$ on the torus $\Sigma \cup \Delta$: given three pairwise disjoint $\Sigma$-rooted tripods, and two disjoint segments of $\boundary(\Delta)$ such that each tripod has exactly one leg in each of them, we can add in $\Delta$ three pairwise disjoint $\Delta$-rooted tripods to embed~$K_{3,3}$. }
    \label{fig:K33onthetorusnew}
\end{figure}

\begin{lem}\label{lemma:outerstringbalancedsep}
    Let~${c,t \geq 1}$ and~${g \geq 0}$ be integers, let~$(\mathcal{S}, \Sigma)$ be a $(g,c)$-outer-string-diagram for a graph~$G$. 
    Let~$s$ be an integer for which~$K_s$ does not embed into~$\Sigma$ and let~$d$ be an integer for which~$K_{3,d}$ does not embed into~$\Sigma$. 
    Let~${m \coloneqq \left\lfloor \frac{1}{3}\binom{s}{2}(2dt-1)\right\rfloor+1}$ and let~${k \coloneqq 4m + 2dtc}$. 
    If~$G$ does not contain a $K_{t+1}$-minor, then for every set of exactly~$csk$ vertices of~$G$ there exists an $\alpha$-balanced separation of order less than~$k$ in $G$, where~${\alpha = \frac{cs-1}{cs}}$. 
\end{lem}

\begin{proof}
    Assume for a contradiction that~$G$ has no $K_{t+1}$-minor and that there is a set~${X \subseteq V(G)}$ of size exactly~${cs(4m + 2dtc)}$ such that $G$ has no $\alpha$-balanced separation of order less than~$k$. 
    
    Since~$\Sigma$ has $c$ cuffs, there exists a cuff $\Delta$ such that~${\lvert X \cap V(\Delta) \rvert \geq s(4m + 2dtc)}$. 
    Divide $\boundary{(\Delta)}$ into~$s$ segments~${I_1, \dots, I_s}$ such that setting~${X_i \coloneqq X \cap V(I_i)}$ for all~${i \in [s]}$, we have~${\lvert X_i \rvert \geq 4m + 2dtc}$.

    \begin{claim}
        For distinct~${i,j \in [s]}$, every~${(I_i,I_j)}$-linkage consisting of paths of length at least~$2$ has size less than~$2dtc$. 
    \end{claim}

    \begin{subproof}
        Assume for a contradiction that there is an $(I_i,I_j)$-linkage consisting of~$2dtc$ paths of length at least~$2$. 
        By the pigeonhole principle, at least $2dt$ of the paths have their second string rooted in the boundary of the same cuff~$\Delta'$. 
        By \Cref{cor:inducedlinkage}, at least~${d}$ of the paths form an induced $(I_i,I_j)$-linkage~$\mathcal{P}$. 

        Now clearly the union of the strings of each of the paths contain a $\Sigma$-rooted tripod with exactly one root contained in~$I_i$, exactly one root contained in~$I_j$, and exactly one root contained in~$\boundary(\Delta') \setminus (I_i \cup I_j)$. 
        Hence, by \cref{lemma:embeddingK3n} there exists a drawing of $K_{3,d}$ on $\Sigma\cup\Delta\cup\Delta'$ without crossings, a contradiction. 
    \end{subproof}
    
    For distinct~${i,j \in [s]}$, by Menger's Theorem, there is an ${(I_i,I_j)}$-linkage~$\mathcal{P}_{i,j}$ of size at least~${4m + 2dtc}$ since, by \cref{obs:balancedseps}, any separation~$(A,B)$ of~$G$ with~${X_i \subseteq I_i \subseteq A}$ and~${X_j \subseteq I_j \subseteq B}$ is $\alpha$-balanced for~$X$. 
    By the claim above, $\mathcal{P}_{i,j}$ contains a $(I_i,I_j)$-linkage~$\mathcal{P}'_{i,j}$ consisting of~${4m}$ paths of length one. 
    For each path~${P \in \mathcal{P}'_{i,j}}$, let~$S_{P,1}$ denote the string rooted in~$I_i$, let~$S_{P,2}$ denote the string rooted in~$I_j$, and let~$S_P$ denote the union of the two strings. 

    Let~$\mathcal{F}$ be a set consisting of exactly one path from each linkage~$\mathcal{P}_{i,j}$. 
    If for each pair~$P,Q$ of distinct paths in~$\mathcal{F}$ we have that~$S_P$ and~$S_Q$ are disjoint, then $\Sigma \cup \Delta$ contains an embedding of~$K_s$ without crossings, obtained by embedding the disjoint union of~$s$ copies of~$K_{1,s}$ into~$\Delta$ where the $i$-th copy of~$K_{1,s}$ is embedded in such a way that the leafs are the roots of the strings in~$P$ that are in~$I_i$.  
    So each such set will contain a pair~$P,Q$ of distinct paths such that~$S_P \cup S_Q$ is connected. 
    We call such a pair \defn{crossing}. 
    Clearly, each crossing pair contains a $\Sigma$-rooted tripod.

    Since each~$\mathcal{P}'_{i,j}$ has size at least $4m$ and each path in $\mathcal{P}'_{i,j}$ has length $1$', we now can greedily choose crossing pairs~$(P_1, Q_1), \dots, (P_m, Q_m)$ such that both~$P_\ell$ and~$Q_\ell$ are vertex-disjoint from~${\bigcup \{ P_z \cup Q_z \colon  z < \ell \}}$. 
    By the pigeonhole principle, there is a pair~${(i,j) \in [s]^2}$ and a set~$\mathcal{T}$ of at least~${2dt}$ $\Sigma$-rooted tripods, contained in the crossing pairs, with exactly one root in~$I_i$ and exactly one root in~$I_j$. 
    Now in the graph obtained by contracting each set of strings in the crossing pairs corresponding to the tripods in~$\mathcal{T}$ to a single vertex, we apply \Cref{thm:duchetmeyniel} to obtain an independent set of size at least~${d}$. 
    Hence, the corresponding subset of~$\mathcal{T}$ contains only pairwise disjoint tripods. 
    By \cref{lemma:embeddingK3n}, we get an embedding of~${K_{3, d}}$ into~$\Sigma\cup \Delta$, a contradiction as before. 
\end{proof}

We prove the following folklore lemma for the sake of completeness.

\begin{lem}
    \label{cor:geni}
    For every integer~${g \geq 0}$, the Euler-genus of~$K_{\lceil \sqrt{6g} \rceil + 5}$ is at least~${g+1}$ and the Euler-genus of~${K_{3,2g+3}}$ is at least~${g+1}$. 
\end{lem}

\begin{proof}
    Since~$K_5$ and~$K_{3,3}$ are non-planar, we may assume that~${g \geq 1}$. 
    By Euler's Formula, for ${n \geq 3}$ the maximum number of edges of an $n$-vertex graph embeddable in a surface of Euler genus~$g$ is~$3(n+g-2)$. 
    For~${s \coloneqq \lceil \sqrt{6g} \rceil + 5}$,
    \begin{align*}
        2|E(K_s)| 
        = 2\mbox{$\binom{s}{2}$} 
        &\geq {(\sqrt{6g} + 5)(\sqrt{6g} + 4)} \\
        &= 6g + 9 \sqrt{6g} + 20 \\
        &= 6g + 6 (\sqrt{6g} + 6) - 16 + 3\sqrt{6g}\\
        &> 6g + 6 (\lceil \sqrt{6g} \rceil + 5) - 12\\
        &= 6( s + g - 2)
    \end{align*}
    Hence, $K_s$ is not embeddable in a surface of Euler-genus~$g$.

    By Euler's Formula, for $n\geq 3$ the maximum number of edges of an $n$-vertex bipartite graph embeddable in a surface of Euler genus~$g$ is~$2(n+g-2)$. 
    For~${d \coloneqq 2g+3}$, 
    \[
       |E(K_{3,d})| = 3d = 6g + 9 > 2 (g + 2g + 3 - 2) = 2 ( d + g - 2 ). 
    \]
    Hence, $K_{3,d}$ is not embeddable in a surface of Euler-genus~$g$. 
\end{proof}

\outerstringgeneral*

\begin{proof}
    Let~${d \coloneqq 2g+3}$ and let~${s \coloneqq \lceil \sqrt{6g} \rceil + 5}$. 
    By \Cref{cor:geni}, the Euler-genus of~$K_{3,d}$ is at least~${g+1}$ and the Euler-genus of~$K_{s}$ is at least~${g+1}$. 
    Let~${t \coloneqq \had(G)}$, let~${m \coloneqq \lfloor \frac{1}{3} \binom{s}{2} (2dt -1) \rfloor + 1}$, let~${k \coloneqq 4m + 2dtc}$, and let~${\alpha \coloneqq \frac{cs-1}{cs}}$. 
    By \Cref{lemma:outerstringbalancedsep}, for every set~$X$ of exactly~${csk}$ vertices of~$G$ there exists an $\alpha$-balanced separation of order less than~$k$ in~$G$. 
    Hence, by \Cref{lemma:balancedSepsForSets}, 
    \begin{align*}
        \tw(G) 
            &< (cs + 1) (4m + 2dtc) \\
            &\leq (c(\sqrt{6g} + 6) + 1) \big( 4(\mbox{$\frac{1}{3}$} \mbox{$\binom{s}{2}$} (2(2g+3)t)) + 2(2g+3)tc \big) \\
            &\leq \big( c(\sqrt{6g} + 6) + 1 \big) \big ( 4 (2g+5)(4g+6)t + c(4g+6)t \big)\\
            &\leq c(8g+20 + c) (4g + 6)(\sqrt{6g} + 7)  t\\
            &\leq ( 80 g^{5/2}c + 224 g^2c + 345 g^{1.5}c^2  +966gc + 960c )t\\
            &\leq ( 1615 g^{5/2} c^2 + 960c ) t.
        \qedhere
    \end{align*}
\end{proof}

\section{Perturbations of circle graphs}
\label{sec:circleperturbations}

Recall that a circle graph is the intersection graph of a set of chords on a circle. We use the following combinatorial description, which allows us to easily generalise the geometric setting to perturbations of circle graphs, as well as streamline the notation in our proofs.

A \defn{cyclic order} on a set $C$ is a ternary relation $R \subseteq C^3$ such that: 
\begin{enumerate}
    [label=(\arabic*)]
   \item if ${(a,b,c) \in R}$, then ${(b,c,a) \in R}$,
   \item if ${(a,b,c) \in R}$, then ${(a,c,b) \notin R}$,
   \item if ${(a,b,c) \in R}$ and ${(a,c,d) \in R}$, then $(a,b,d) \in R$, and
   \item for all distinct $a,b,c$, either $(a,b,c) \in R$ or ${(a,c,b) \in R}$.
\end{enumerate}
These axioms imply that if  $(a,b,c)\in R$ then $a,b,c$ are distinct. For example, if $(a,a,b) \in R$, then $(b,a,a) \in R$ by (1), and then $(b,a,a) \notin C$ by (2), which is a contradiction. 
By abuse of notation, we henceforth use~$C$ to denote both the ground set and the relation of the cyclic order, 
and also write~${(x_1, \dots, x_n) \in C}$ for a tuple~$(x_1, \dots, x_n)$ of distinct elements of the ground set if~${n \geq 3}$ and~${(x_i,x_j,x_k) \in C}$ whenever~${i<j<k}$. 
An \defn{interval} of~$C$ is a tuple~$(x_1, \dots, x_n) \in C$ with~$n \geq 3$ such that for every~$x$ in the ground set of $C$, if~$(x_1, x, x_n) \in C$, then~${x \in \{x_1,\dots,x_n\}}$. 

A \defn{chord} of~$C$ is a set of two distinct groundset elements of~$C$, which we refer to as the \defn{endpoints} of the chord. 
Two chords are \defn{independent} if their endpoints are distinct.
Two independent chords~${\{x,y\}}$ and~${\{x',y'\}}$ \defn{cross} if~${(x,x',y,y') \in C}$. 
A \defn{chord-diagram} is a tuple~$(C,D)$ consisting of a cyclic order~$C$ and a set~$D$ of pairwise independent chords of~$C$. 
The \defn{crossing graph} of a chord-diagram~$(C,D)$ is the graph with vertex set~$D$ where two chords in $D$ are adjacent if they cross. 
We also call this graph the  \defn{circle graph represented by~$(C,D)$}.  
Without loss of generality, we will always assume that each~${x \in C}$ is an endpoint of some chord in~$D$.

For a graph~$G$, consider a colouring~${\zeta \colon V(G) \to [k]}$, as well as an auxiliary graph~$H$ on~$[k]$ which may contain loops but no parallel edges. 
We call the tuple~$(H,\zeta)$ a \defn{$k$-colour-perturbation-model} of~$G$. 
The \defn{$(H,\zeta)$-perturbed} $G$ is the graph with vertex set~$V(G)$ where two vertices~$v$ and~$w$ are adjacent if either: 
\begin{itemize}
   \item ${vw \in E(G)}$ and~${\zeta(v) \zeta(w) \notin E(H)}$, or
   \item ${vw \notin E(G)}$ and~${\zeta(v) \zeta(w) \in E(H)}$. 
\end{itemize}

For an integer~$r\geq 0$, a graph~$G$ is a \defn{rank-$r$-perturbation of a graph~$G_0$} if the binary adjacency matrix of~$G$ is equal to the matrix obtained from the sum of a binary adjacency matrix of~$G_0$ with a symmetric ${(\lvert V(G) \rvert \times \lvert V(G) \rvert)}$-matrix~$P$ of rank~$r$, which we call the \defn{perturbation matrix}, by changing all diagonal entries to~$0$~\cite{McCarty21}.

A rank-$1$-perturbation of a graph~$G$ corresponds exactly to the graph obtained from~$G$ by complementing the edges on some set of vertices. 
Moreover, the complementation of the edges between two disjoint sets~$X$ and~$Y$ of vertices can be modelled by three rank-$1$-perturbations: complement on~${X \cup Y}$, on~$X$, and on~$Y$. 
A rank-$r$-perturbation can be understood by successively performing a bounded number of rank-$1$-perturbations, as  illustrated in the following lemma. 

\begin{lem}
    \label{lem:rank-r-pertubation}
    If a graph~$G$ is a rank-$r$-perturbation of a graph~$G_0$, then there exists a $k$-colour-perturbation-model~$(H,\zeta)$ for some~${k \leq 2^r}$, such that~$G$ is the~${(H,\zeta)}$-perturbed~$G_0$. 
\end{lem}

\begin{proof}
    Let~$A$ denote the binary adjacency matrix of~$G$ that is obtained from~${A_0 + P}$ by changing all diagonal entries to~$0$, where~$A_0$ is a binary adjacency matrix of~$G_0$ and~$P$ is a symmetric~${(n \times n)}$-matrix of rank~$r$. 
    Let~${( v_i \colon i \in [n] )}$ be the enumeration of~$V(G)$ given by the ordering of~$A$. 
    Note that the column space~${\{ p_j \colon j \in [2^r] \}}$ of~$P$ (as a subspace of~$\GF(2)^n$) has size~$2^r$. 
    Let~$\zeta$ denote the map from~$V(G)$ to~$[2^r]$ that maps~$v_i$ to the unique~$j$ so that the $i$-th column of~$P$ equals~$p_j$. 
    We define the auxiliary graph~$H$ on~$[2^r]$ by letting~${j j'}$ be an edge of~$H$ if and only if there exists~${i, i' \in [n]}$ for which~${\zeta(v_i) = j}$ and~${\zeta(v_{i'}) = j'}$ such that the $i'$-th entry of~$p_i$ is equal to~$1$. 
    It is straight-forward to verify that, since~$P$ is symmetric, this definition is well-defined, and that the $(H,\zeta)$-perturbed~$G_0$ is isomorphic to~$G$. 
\end{proof}

A graph~$G$ is \defn{cyclically ordered} if~$V(G)$ is the ground set of a cyclic order~$C$. 
Analogously to ordered graphs, we say that two non-incident edges~$uv$ and~$xy$ of~$G$ \defn{cross} if~${(u,x,v,y) \in C}$. 

\begin{lem}
    \label{lemma:goodcolour}
    For any integer~$k\geq 1$, given a cyclic order~$C$ on~$V(K_{4k})$ and a partition~$(E_i \colon i \in [k])$ of~$E(K_{4k})$ there is some~${i \in [k]}$ such that the subgraph ${G_i \coloneqq (V(K_{4k}), E_i)}$ contains both: 
    \begin{itemize}
        \item a pair of non-incident edges that cross, and
        \item a pair of non-incident edges that do not cross. 
    \end{itemize}
\end{lem}

\begin{proof}
    Note that since an outerplanar graph on~$n$ vertices has at most~${2n-3}$ edges, if~${\lvert E_i \rvert \geq 8k-2}$ for some~${i \in [k]}$, then~$G_i$ contains a pair of edges which cross. 

    Suppose~${\lvert E_i \rvert \geq 4k+1}$.
    Consider the subgraph~$G'_i$ obtained from~$G_i$ by iteratively deleting vertices of degree less than two. 
    Observe that
    \[
        \lvert E(G'_i) \rvert 
        \geq \lvert E(G_i) \rvert  - \big( \lvert V(G_i) \rvert - \lvert V(G'_i) \rvert \big)
        > \lvert V(G'_i) \rvert.
    \]
    In particular, $G'_i$ has a vertex~$v$ of degree at least~$3$. 
    Let~${x,y,z \in V(G'_i)}$ be neighbours of $v$ such that~${(v,x,y,z) \in C}$. Since~$G'_i$ has minimum-degree at least two, $y$ has a neighbour~$w$ distinct from~$v$. 
    Since~$C$ is a cyclic order, one of~$(y,w,v)$ and~$(v,w,y)$ is in~$C$. 
    If~${(y,w,v) \in C}$, then~$yw$ and~$vx$ is a pair of non-incident edges which do not cross. 
    If~${(v,w,y) \in C}$, then~$yw$ and~$vz$ is a pair of non-incident edges which do not cross. 

    Now since~${\lvert E(K_{4k}) \rvert = \binom{4k}{2} = 2k(4k-1) = k(8k - 2)}$, 
    by the pigeonhole principle there is an~${i \in [k]}$ such that~${\lvert E_i \rvert \geq 8k-2 \geq 4k+1}$. 
    By the observations above, $G_i$ contains both of the desired pairs of edges. 
\end{proof}

\begin{lem}
    \label{lemma:kcirclegraphbalancedsep}
    Let~$G_0$ be a circle graph, let~$(H,\zeta)$ be $k$-colour-perturbation-model of~$G_0$ for some integer~${k \geq 1}$, and let~$G$ be the $(H,\zeta)$-perturbed~$G_0$. 
    If~$G$ does not contain a $K_{t+1}$-minor, then for every set~$X$ of exactly~${4k^2(4k+9)t}$ vertices~$G$ there exists an $\alpha$-balanced separation of order less than~${{k (4k + 9)t}}$ in $G$, where~${\alpha = \frac{4k-1}{4k}}$. 
\end{lem}

\begin{proof}
    Let~$(C,D)$ be the chord-diagram of~$G_0$ and let~${s \coloneqq 4k}$. 
    Assume that there is a set~$X$ of exactly~${sk(4k+9)t}$ vertices of~$G$ such that $G$ has no $\alpha$-balanced separation of order less than~${k (4k + 9)t}$ for~$X$. 
    We show that~$G$ contains a $K_{t+1}$-minor.

    Let~$C = ( v_1, \dots, v_{2n} )$ and let~$\Root(X)$ denote the set of endpoints of the chords in~$X$. 
    Note that ${\lvert \Root(X) \rvert = 2sk(4k+9)t}$. 
    We partition~$C$ into~$s$ intervals~${I_j = (v_{\ell_j}, v_{\ell_j+1}, \dots,  v_{r_j})}$ such that~$\ell_1 = 1$,~${\ell_{j+1} = r_j +1}$, and~${\lvert I_j \cap \Root(X) \rvert = 2k(4k+9)t}$ for each~${j \in [4k]}$.
    Let~$\mathcal{I}$ denote the set of these intervals. 
    For each~${j \in [s]}$, let~$D_j$ denote the set of chords in~$D$ that have at least one endpoint in~$I_j$, and let~${X_j \coloneqq D_j \cap X}$. 
    Note that~${k(4k+9)t \leq \lvert X_j \rvert \leq 2k(4k+9)t}$. 

    We prove the following claim about induced linkages between these sets.

    \begin{claim}
        Let~${a,b \in [s]}$ with~${a < b}$. 
        Every induced $(D_a \setminus D_b, D_b \setminus D_a)$-linkage in ${G - (D_a \cap D_b)}$ has size at most~${k(2k+4)}$. 
    \end{claim}
    
    \begin{subproof}
        Suppose for contradiction that there is an induced $(D_a \setminus D_b, D_b \setminus D_a)$-linkage in~${G - (D_a \cap D_b)}$ of size at least~${k(2k+4)+1}$.
        By the pigeonhole principle, there is an induced $(D_a \setminus D_b, D_b \setminus D_a)$-linkage~$\mathcal{P}$ in~${G - (D_a \cap D_b)}$ of size at least~${2k + 5}$ such that for some~${c_1 \in [k]}$, we have~${\zeta(v) = c_1}$ for each~$v$ that is a start vertex in~$D_a$ of some path in~$\mathcal{P}$. 
        Note that for some interval $I\in \{(v_{r_a + 1}, \dots, v_{\ell_b - 1}),(v_{r_b + 1}, \dots v_{2n}, v_1, \dots, v_{\ell_a - 1})\}$, there are at least~${k + 3}$ paths in~$\mathcal{P}$ such that the endpoint of their first chord that is not contained in~$I_a$ is contained in~$I$.  
        Let~$\{x_0, y_0\}, \{x_1, y_1\}, \dots, \{x_{k+2}, y_{k+2}\}$ be the chords corresponding to the start vertices of the paths~${P_0, \dots, P_{2k+2} \in \mathcal{P}}$, such that~$(x_0,\dots,x_{k+2}) \in C$ are the endpoints in~$I_a$ and either~$(y_0, \dots, y_{k+2}) \in C$ or~$(y_{k+2}, \dots, y_0) \in C$ are the endpoints in~$I$ (depending on whether~$c_1$ has a loop in~$H$).
        Let~$B$ denote the interval of~$C$ between~$y_{0}$ and~$y_{k+2}$ that contains~$y_1$, and for each $m\in [k+1]$ let~$\{x'_m,y'_m\}$ denote the first chord in~$V(P_m)$ after~$\{x_m,y_m\}$ that does not have both endpoints in~$B$. 
        Note that such a chord exists since the final chord in~$V(P_m)$ has an endpoint in~$I_b$, which by construction is disjoint from~$B$. 
        Again by the pigeonhole principle, there are~${p,q \in [k+1]}$ with~${p < q}$ and an~${c_2 \in [k]}$ such that~${\zeta(\{ x'_p, y'_p\}) =  \zeta(\{ x'_q, y'_q\}) = c_2}$. 
        
        If for~${i \in \{p,q\}}$ we have that~$\{x'_i,y'_i\}$ has exactly one endpoint in~$B$, then it crosses exactly one of $\{x_{0},y_{0}\}$ and~$\{x_{k+2},y_{k+2}\}$. 
        This contradicts that~$\{P_{0}, P_i, P_{k+2}\}$ is an induced linkage. 
        
        So for each~${i \in \{p,q\}}$, we may assume that both~$x'_i$ and~$y'_i$ are not in~$B$. 
        By the choice of~${\{x'_{i},y'_{i}\}}$, each internal vertex of the subpath of~$P_{i}$ between~$\{x_{i},y_{i}\}$ and~$\{x'_{i},y'_{i}\}$ has both their endpoints in~$B$. 
        Let~${\{x''_i, y''_i\}}$ denote the neighbour of~${\{x'_{i},y'_{i}\}}$ on this subpath.
        Note that~${\{x'_p,y'_p\}}$ crosses~${\{x_p,y_p\}}$ if and only if it also crosses both~${\{x_0,y_0\}}$ and~${\{x_{k+2},y_{k+2}\}}$.
        Since~${\{x_0,y_0\}}$ and~${\{x'_p,y'_p\}}$ are not adjacent, we deduce that~${\{x''_p, y''_p\} \neq \{x_p,y_p\}}$.
        Hence,~${\{x''_p, y''_p\}}$ has both endpoints in~$B$, and so does not cross~$\{x'_{p},y'_{p}\}$. 
        With~${c_3 \coloneqq \zeta(\{x''_p, y''_p\})}$, we observe that~${c_2 c_3}$ is an edge of~$H$. 
        But by the same argument,~${\{x''_p, y''_p\}}$ does not cross~$\{x'_{q},y'_{q}\}$. 
        Thus~${\{x''_p, y''_p\}}$ and~$\{x'_{q},y'_{q}\}$ are adjacent, contradicting the fact that~$\mathcal{P}$ is an induced linkage.
    \end{subproof}

    For any two distinct~${a,b \in [s]}$, by Menger's Theorem, there is a ${(D_a,D_b)}$-linkage of size at least~${t {k (4k + 9)} = 2tk(2k+4) + tk}$ since, by \Cref{obs:balancedseps}, any separation~${(A,B)}$ of~$G$ with~${X_a \subseteq D_a \subseteq A}$ and~${X_b \subseteq D_b \subseteq B}$ is $\alpha$-balanced for $X$. 
    We can deduce from \Cref{cor:inducedlinkage} and the above claim that~$G$ either contains the desired $K_{t+1}$-minor, or that~$G$ does not contain a ${(D_a \setminus D_b,D_b \setminus D_a)}$-linkage of size greater than~${2tk(2k+4)}$.
    Assuming the latter holds, ${\lvert D_a \cap D_b \rvert \geq tk}$, and in particular~${D_a \cap D_b}$ contains a monochromatic subset~$F_{a,b}$ of size at least~$t$. 
    Consider the cyclic order~$C'$ on~$[s]$ with~${(1, \dots, s) \in C'}$. 
    We colour the edges of the cyclically ordered~$K_{s}$ with~$[k]$ colours by assigning to an edge~$ab$ with~${a < b}$ the colour of~$F_{a,b}$. 
    Note that if two non-incident edges~$ab$ and~$a'b'$ of this auxiliary graph cross, then every chord in~$F_{a,b}$ crosses every chord in~$F_{a',b'}$, and if they do not cross, then no edge in~$F_{a,b}$ crosses any edge in~$F_{a',b'}$. 
    By \Cref{lemma:goodcolour}, there is an~${i \in [k]}$, two non-incident edges~$a_1 b_1$ and~$a'_1 b'_1$ that cross, and two non-incident edges~$a_2 b_2$ and~$a'_2 b'_2$ that do not cross. 
    Hence, depending on whether the colour of these chords has a loop in~$H$, either~${G[F_{a_1,b_1} \cup F_{a'_1,b'_1}]}$ or~${G[F_{a_2,b_2} \cup F_{a'_2,b'_2}]}$ contains the complete bipartite graph~$K_{t,t}$ as subgraph. 
    Contracting all but one edge from a perfect matching of this subgraph yields the desired $K_{t+1}$-minor.     
\end{proof}

\begin{thm}
\label{SomeTheorem}
    For every circle graph~$G_0$ and every $k$-colour-perturbation-model~$(H,\zeta)$ of~$G_0$, the graph~$G$ which is the $(H,\zeta)$-perturbed~$G_0$ satisfies 
    \begin{align*}
        \tw(G) 
            \leq (16k^3 +40 k^2 +9 k) \had(G)
            \leq 65 \cdot k^{3} \cdot \had(G).
    \end{align*}
\end{thm}

\begin{proof}
    Let~${t \coloneqq \had(G)}$ and let~${\alpha \coloneqq \frac{4k-1}{4k}}$. 
    By \Cref{lemma:kcirclegraphbalancedsep}, for every set of exactly~${4tk^2(4k+9)}$ vertices of~$G$ there exists an $\alpha$-balanced separation of order less than~${tk (4k + 9)}$ in $G$. 
    Hence, by \Cref{lemma:balancedSepsForSets}, 
    \[
        \tw(G) 
            \leq (4k + 1)tk(4k+9)
            \leq (16k^3 +40 k^2 +9 k) t
            \leq 65 k^3 t.
        \qedhere
    \]
\end{proof}

\Cref{thm:c-pertubationofcirclegraph} immediately follows from \cref{lem:rank-r-pertubation,SomeTheorem}.

\renewcommand{\thefootnote}{\arabic{footnote}}

\subsection*{Acknowledgements} 

This research was initiated at the
\href{https://www.matrix-inst.org.au/events/structural-graph-theory-downunder-III/}{Structural Graph Theory Downunder III} program of the Mathematical Research Institute MATRIX (April 2023).

{
\fontsize{10pt}{11pt}
\selectfont
\bibliographystyle{DavidNatbibStyle}
\bibliography{DavidBibliography}
}

\begin{appendices}
\section{}

Here we prove the following result promised in \cref{sec:InducedGridMinors}.

\begin{thm}
   For every positive integer~$t$, every graph that contains a ${(2t\times (4t-1))}$-grid as an induced minor contains a planar induced subgraph with treewidth at least~$t$. 
\end{thm}

\begin{proof}
    If a graph~$G$ contains a ${(2t \times (4t-1))}$-grid as an induced minor, then there is a set ${\{H_{i,j} \colon (i,j) \in [2t]\times [4t-1]\}}$ of pairwise disjoint connected subgraphs of~$G$ such that there is an edge of~$G$ between distinct~$H_{i,j}$ and~$H_{i',j'}$ if and only if~${\{|i-i'|,|j-j'|\} = \{0,1\}}$.
    Let 
    \begin{align*}    
        S_1 &\coloneqq[2t]\times \{1,4t-1\}, \\
        S_2 &\coloneqq\{(i,j)\in [2t]\times [4t-1] \colon i \textrm{ is odd}\}, \\
        S_3 &\coloneqq\{(i,j)\in [2t]\times [4t-1] \colon i\equiv 2\mod 4,\quad j\equiv 3\mod 4\}, \\
        S_4 &\coloneqq\{(i,j)\in [2t]\times [4t-1] \colon i\equiv 0\mod 4,\quad j\equiv 1\mod 4\},\\
        S &\coloneqq S_1\cup S_2\cup S_3\cup S_4.
    \end{align*}
    The minor~$H$ obtained from~$G$ by deleting all vertices not in~${\bigcup \{ V(H_{i,j}) \colon (i,j) \in S \}}$ and then contracting all edges in~${\bigcup \{ E(H_{i,j}) \colon (i,j) \in S \}}$ has maximum degree~$3$, and additionally every neighbour of a degree~$3$ vertex in~$H$ has degree~$2$. 
    It is easily seen that~$H$ contains a ${t \times t}$-grid as a minor, and therefore has treewidth at least~$t$ (see \citep{HW17} for example). 
    We now show that~$G$ has a planar induced subgraph which contains~$H$ as a minor.
    
    Let~${\{X_{i,j} \colon (i,j) \in S\}}$ be a collection of subsets of~$V(G)$ with the following properties.
    \begin{enumerate}
        \item\label{item:appendix1} for all $(i,j)\in S$, and $G[X_{i,j}]$ is a connected subgraph of $V(H_{i,j})$,
        \item\label{item:appendix2} the graph obtained from $G$ by deleting all vertices not in $\bigcup \{X_{i,j} \colon (i,j) \in S\}$ and then contracting all edges in $\bigcup \{E(G[X_{i,j}]) \colon (i,j)\in S\}$ is $H$,
        \item\label{item:appendix3} subject to the other two conditions $|\bigcup \{X_{i,j} \colon (i,j)\in S\}|$ is as small as possible.    
    \end{enumerate}
    Consider a set~$X_{i,j}$ corresponding to a degree-$2$ vertex of~$H$, and let~$X_{a,b}$ and~$X_{c,d}$ be the sets corresponding to its neighbours. 
    The minimality condition ensures that~${G[X_{i,j}]}$ is exactly an ${(N_G(X_{a,b}),N_G(X_{c,d}))}$-path. 
    Now consider a set~$X_{i,j}$, and let~$u$, $v$, and~$w$ be the three unique vertices in~$N_G(X_{i,j})$ in sets corresponding to neighbours of~$X_{i,j}$ in~$H$.
    Let~$H'$ be the graph obtained from~${G[X_{i,j} \cup \{u,v,w\}]}$ by adding a new vertex~$z$ adjacent to~$u$, $v$, and~$w$. 
    If~$H'$ contains~$K_{3,3}$ or~$K_5$ as a minor, then we can easily find a proper subset of~$X_{i,j}$ which maintains Properties \ref{item:appendix1} and \ref{item:appendix2}, contradicting Property~\ref{item:appendix3}.
    Thus by Wagner's Theorem, $H'$ is planar, which means~${G[X_{i,j} \cup \{u,v,w\}]}$ can be drawn in the plane with~$u$, $v$, and~$w$ on the outer face.
    Thus~${G[\bigcup \{X_{i,j} \colon (i,j) \in S\}]}$ is the desired induced planar subgraph. 
\end{proof}
\end{appendices}
\end{document}